\documentclass[11pt]{amsart}
\usepackage{geometry}
\usepackage[latin1]{inputenc} 
\usepackage[T1]{fontenc}
\usepackage{graphicx}
\usepackage{amssymb}
\usepackage{epstopdf}
\usepackage{enumitem}
\usepackage{mathtools}
\usepackage{bbm}

\theoremstyle{plain}
\newtheorem{theorem}{Theorem}[section]
\newtheorem{corollary}[theorem]{Corollary}

\newtheorem{lemma}[theorem]{Lemma}
\newtheorem{proposition}[theorem]{Proposition}
\theoremstyle{remark}
\newtheorem{remark}[theorem]{Remark}

\theoremstyle{definition}
\newtheorem{definition}[theorem]{Definition}

\numberwithin{equation}{section}

\title[Bounds on Mixing Time for Time-Inhomogeneous Markov Chains]{Bounds on Mixing Time for Time-Inhomogeneous \\ Markov Chains}
\author{Raphael Erb}

\begin{document}
\maketitle

\begin{abstract}
Mixing of finite time-homogeneous Markov chains is well understood nowadays, with a rich set of techniques to estimate their mixing time. In this paper, we study the mixing time of random walks in dynamic random environments. To that end, we propose a concept of mixing time for time-inhomogeneous Markov chains.
We then develop techniques to estimate this mixing time by extending the evolving set method of \cite{Morris2003EvolvingSM}. We apply these techniques to study a random walk on a dynamic Erd\H{o}s-R\'enyi graph, proving that the mixing time is $O(\log(n))$ when the graph is well above the connectivity threshold. We also give an almost matching lower bound.
\end{abstract}

\section{Introduction}

In recent years, several papers have been published on the topic of mixing time of random walks in dynamic random environments (refer to \cite{10.1214/17-AAP1289}, \cite{10.1007/978-3-030-54921-3_7}, \cite{10.1214/20-AOP1441}, \cite{quenchedExitTimesPeresSousiSteif}, \cite{Peres:2020ub},  for a selection).
One implicit technical challenge these works have faced is that random walks in dynamic environments, in general, do not have a stationary distribution, which is required to define the mixing time. A common way around this issue used in some of the above references is to consider models where, despite the dynamic environment, the stationary distribution exists and is independent of time.

The goal of this paper is to consider situations where such time-independent stationary distributions do not exist. We will propose a candidate for a ``time-dependent'' stationary distribution and a corresponding definition of mixing time.
To show the workability of this approach, we develop a method to find bounds on this new mixing time, based on the theory of evolving sets (see \cite{Morris2003EvolvingSM}). This method has already found use for both static environments as well as, more recently, dynamic graphs with time-independent invariant measure (see~\cite{Peres:2020ub}).

To describe the main idea of our work, let us recall the usual definition of mixing time.
For a Markov chain $X$ on a finite state space with (time-independent) transition matrix $P$ and unique invariant distribution $\pi$, the $\varepsilon$-mixing time, $\varepsilon \in (0,1)$, is defined by
\begin{equation}\label{mixingdefstat}
t^\mathrm{static}_\mathrm{mix}(\varepsilon) = \inf \{t \geq 0 : \sup_x \|P^t(x, \cdot) - \pi(\cdot)\|_\mathrm{TV} \leq \varepsilon \}
\end{equation}
where $\|\mu - \nu\|_\mathrm{TV} = \frac{1}{2} \sum_x |\mu(x) - \nu(x)|$ is the total variation distance and $P^t$ denotes the $t$-th power of $P$. That is to say the mixing time is the first time where the distribution of $X_t$ is $\varepsilon$-close to $\pi$, uniformly in the starting position.

We recall furthermore (see e.g.~Chapter 4 of \cite{LevinPeresWilmer2006}) that 
\begin{equation}\label{triangle}
\sup_x \|P^t(x,\cdot) - \pi(\cdot)\|_{\mathrm{TV}} \leq \sup_{x,y} \|P^t(x,\cdot) - P^t(y,\cdot)\|_{\mathrm{TV}} \leq 2 \cdot \sup_x \|P^t(x,\cdot) - \pi(\cdot)\|_{\mathrm{TV}}
\end{equation}
and hence $t^\mathrm{static}_\mathrm{mix}$ is related to
\begin{equation}\label{mixingdefalt}
\tilde{t}^\mathrm{static}_{\mathrm{mix}}(\varepsilon) \coloneqq \inf \{t \geq 0 : \sup_{x,y} \|P^t(x,\cdot) - P^t(y,\cdot)\|_\mathrm{TV} \leq \varepsilon\}
\end{equation} 
by the inequality
\begin{equation}\label{mixingdefaltineq}
t^\mathrm{static}_{\mathrm{mix}}(2\varepsilon) \le \tilde{t}^\mathrm{static}_{\mathrm{mix}}(2\varepsilon) \le t^\mathrm{static}_{\mathrm{mix}}(\varepsilon), \qquad \text{for every } \varepsilon \le \frac{1}{2}.
\end{equation}

Let us come to the setting of this paper and consider a discrete time-inhomogeneous Markov chain $(X_t)_{t \in \mathbb{Z}}$ on $[n] = \{1, \dots, n\}$ for $n \in \mathbb{N}$ with time-dependent transition matrices $P~=~(P_t)_{t \in \mathbb{Z}}.$ 
We denote by $\mathbb{P}^P$ the distribution of this chain, that is to say
\begin{equation}\label{transitionMatrices}
\mathbb{P}^P(X_{t} = y | X_{t-1} = x) = P_{t}(x,y).
\end{equation}
For $s < t$, we write $P^{s,t} = P_{s+1}\cdots P_t$ for the matrix product. 

In the applications we have in mind, in particular for random walks on dynamic random graphs, the transition matrices $(P_t)_{t \in \mathbb{Z}}$ will themselves be random. As we will aim for quenched estimates on the mixing time, we first consider $(P_t)_{t \in \mathbb{Z}}$ deterministic, and only in Section~\ref{randEnvSection} we introduce a probability distribution on the set of transition matrices.

In order to define mixing time for $(X_t)_{t \in \mathbb{Z}}$, we replace $\pi$ in (\ref{mixingdefstat}) with a suitable time-dependent distribution $\pi_t$. Given the sequence $(P_s)_{s \in \mathbb{Z}}$, for $t \in \mathbb{Z}$ and for some arbitrary state $x$, set
\begin{equation}\label{pitdef}
\pi_t(\cdot) \coloneqq \lim_{s \to -\infty} P^{s,t}(x,\cdot) = \lim_{s \to -\infty} \mathbb{P}^P(X_t = \cdot | X_s = x)
\end{equation}
if the limit exists. Given some trivial irreducibility assumptions we will see that $\pi_t(\cdot)$ does not depend on the starting state $x$. 

Let $\varepsilon \in (0,1)$, $s \in \mathbb{Z}$. The $\varepsilon$-mixing time for a sequence $P = (P_t)_{t \in \mathbb{Z}}$ at time $s \in \mathbb{Z}$ is
\begin{equation}\label{mixingdef}
t_{\mathrm{mix}}^{P}(\varepsilon,s) \coloneqq \inf \{t \geq 0: \sup_x \|P^{s,s+t}(x,\cdot) - \pi_{s+t}(\cdot)\|_\mathrm{TV} \leq \varepsilon \}.
\end{equation}

In principle we could take an alternative approach and use a definition similar to (\ref{mixingdefalt}) even for a time-inhomogeneous chain. By (\ref{mixingdefaltineq}) it would be essentially equivalent to (\ref{mixingdef}), but we would circumvent the construction of $\pi_t$. However, our approach using $\pi_t$ has some advantages that allow us to apply methods from the time-homogeneous situation. In particular the evolving sets method uses a target distribution $\pi$ (or $\pi_t$ in our case) that is approached by the Markov chain.

In this paper, we are going to show a generalization of Theorem 17.10 in~\cite{LevinPeresWilmer2006} for time-inhomogeneous chains. In the case of time-homogeneous Markov chains, Jerrum and Sinclair showed in~\cite{JerrumSinclair} that a reversible chain on some graph with stationary distribution $\pi$ has mixing time less than $2\Phi_*^{-2}(\log \varepsilon^{-1} + \log \frac{1}{\min_{z} \pi(z)}) $ where $\Phi_*$ is the \emph{bottleneck ratio} or \emph{conductance} of the chain. This was then soon generalized to non-reversible chains (see \cite{Mihail1989ConductanceAC}). The introduction of evolving sets more than a decade later led to a new way of proving those same bounds (see Theorem 17.10 in the book \cite{LevinPeresWilmer2006} for a precise statement and the full proof). We will discuss this method in more detail in Section~\ref{evSetsSection}.

\subsection{Literature}

There are other possible approaches to define mixing time for time-inhomogeneous Markov chains with no time-independent stationary distribution. One such approach we want to briefly discuss is found in~\cite{10.1007/978-3-030-54921-3_7} where the authors study a random walk on a randomly evolving graph where the edges appear and disappear independently. They propose the following notion of mixing time: At each time $t$, consider the stationary distribution $\tilde{\pi}_t$ of a time-homogeneous Markov chain with transition matrix $P_t$, and then observe whether the distribution of the time-inhomogeneous Markov chain stays close to those different $\tilde{\pi}_t$. A Markov chain on a graph with $n$ vertices is said to have mixed if it did stay close to $\tilde{\pi}_t$ for at least $\sqrt{n}$ consecutive time steps. The choice of $\sqrt{n}$ is essentially arbitrary and might be considered artificial. We do not need to make such a choice, since we will see that the distance $\sup_x \|P^{s,s+t}(x,\cdot) - \pi_{s+t}(\cdot)\|_\mathrm{TV}$ for our proposed $\pi_t$ is non-increasing, which implies that a mixed chain will remain mixed.

However, even though there are fundamental differences between the definition of mixing time in~\cite{10.1007/978-3-030-54921-3_7} and the definition in the present paper, the example in Section~\ref{exampleSection} yields the same order of upper bound for mixing time as Theorem 1.2 and Theorem 1.3 in~\cite{10.1007/978-3-030-54921-3_7}.

\subsection{Organisation}
This paper is divided into four further sections: In Section 2, we define the proposed notion of mixing time. In particular, we give criteria under which our time-dependent target distribution exists. Technical proofs are deferred to Appendix~\ref{appendix:existence}.
In Section 3, we develop techniques for proving upper bounds on mixing time and we adapt the theory of evolving sets to the time-inhomogeneous setting. in~\cite{Peres:2020ub}, by assuming the existence of a time-independent stationary distribution, some generalizations have been achieved. We show that our time-dependent target measure is a suitable replacement for the stationary distribution to derive analogous results. We use evolving sets to prove a generalization of Theorem 17.10 in~\cite{LevinPeresWilmer2006} for time-inhomogeneous Markov chains.
In Section~4, we move from deterministic time-inhomogeneous Markov chains to chains where the transition matrix for each time step is random.

Finally in Section 5, we present a concrete example: A random walk on a randomly evolving Erd\H{o}s-R\'enyi graph where each graph is independent of its predecessor. Assuming that each graph is very likely to be above the connectivity threshold (that is, each vertex has more than $c_1 \log n$ neighbours for some $c_1$ large enough), we show that the mixing time of the random walk is $O(\log n)$. This is a result that is well-known for static Erd\H{o}s-R\'enyi graphs.

\section{Target distribution and mixing time}\label{definitionsSection}

\subsection{Target distribution} In this section, we will define the target distribution $\pi_t$ and give an irreducibility condition that is sufficient to show its existence.

Without loss of generality, we consider Markov chains on the state space $[n] \coloneqq \{1, \dots, n\}$
for a fixed $n \in \mathbb{N}.$ Furthermore, let $P = (P_s)_{s \in \mathbb{Z}}$ be a sequence of transition matrices on~$[n]$ and $X=~(X_s)_{s \in \mathbb{Z}}$ the time-inhomogeneous Markov chain governed by the transition matrices, as in~(\ref{transitionMatrices}). Let $\mathbb{P}_{x,s}^{P}$ be the distribution of $(X_t)_{t \ge s}$ when started from $x$ at time $s$.

We define a sequence of probability measures on $[n]$ that takes on the role of a stationary distribution for the Markov chain $X$.  Set for $x,y \in [n]$, $t \in \mathbb{Z}$, 
\begin{equation} \label{QtDefinition}
\lim_{s \to -\infty} P^{s,t}(x,y) = \lim_{s \to -\infty} \mathbb{P}^P_{x,s}(X_t = y) \eqqcolon Q^t(x,y)
\end{equation}
for some limiting matrix $Q^t$, if the limit exists. 

If the value of $Q^t(x,y)$ does not depend on $x$, i.e.~$Q^t$ is a rank one matrix, we set $\pi_t(y) = Q^t(x,y)$ and say that $\pi_t$ exists.  To formulate conditions for this to happen we introduce a quantity which measures the similarity between rows of a stochastic matrix $P$,
\begin{equation}
\delta(P) \coloneqq \sup_{x,y} \sum_{z} [P(x,z) - P(y,z)]^{+} = \sup_{x,y} \|P(x,\cdot) - P(y,\cdot)\|_\mathrm{TV} \in [0,1]
\end{equation}
where $[a]^{+} \coloneqq \max(0,a).$ $\delta(P) = 0$ if and only if all rows of $P$ are equal. $\delta$ is known as Dobrushin's ergodic coefficient (see \cite{dobrushin1956}). 

The question whether a matrix $Q^t$ as in (\ref{QtDefinition}) exists is then related to the concepts of weak ergodicity (rows of $P^{s,t}$ are approaching each other as $t \to \infty$) and strong ergodicity ($P^{s,t}$ converges to a limit as $t \to \infty$) that are described in e.g.~ Chapter 12 of \cite{bremaud}, with one fundamental difference: By taking the limit $s \to -\infty$ instead of $t \to \infty$, weak and strong ergodicity are equivalent. In that vein, we only need to find conditions such that $\lim_{s \to -\infty} \delta(P^{s,t}) = 0$ and can deduce the existence of the limit $Q^t$ from that. The following proposition provides such conditions:

\begin{proposition} \label{deterministicExistenceProp} Let $(P_s)_{s \in \mathbb{Z}}$ be a sequence of stochastic matrices. For $t \in \mathbb{Z}$, $\varepsilon > 0$, iteratively define $t_0 = t$ and
$t_k = \sup \{t_k < t_{k-1} : \delta(P^{t_k, t_{k-1}}) \leq 1 - \varepsilon\}$, terminating when $t_k = -\infty$. Assume that for every $t \in \mathbb{Z}$, there exists $\varepsilon > 0$ such that
$\{t_1, t_2, ...\}$ is an infinite set. Then, for every $t \in \mathbb{Z}$, $\lim_{s \to -\infty} P^{s,t} \eqqcolon Q^t$ exists, $\delta(Q^t) = 0$, $\pi_t(y) \coloneqq Q^{t}(1,y)$ for every $y \in [n]$ is well-defined.
\end{proposition}

The proof of this proposition follows immediately from Lemma~\ref{deltasubmult} and Lemma~\ref{Lemmaexist} below. These lemmas are standard results. For Lemma~\ref{deltasubmult} we refer to \cite{paz_reichaw_1967}. The proof of Lemma~\ref{Lemmaexist} is given in Appendix~\ref{appendix:existence}.

\begin{lemma}[Submultiplicativity]\label{deltasubmult}
For any two stochastic matrices $P,Q$ it holds that 
$$\delta(PQ) \leq \delta(P)\delta(Q).$$
\end{lemma}

\begin{lemma}\label{Lemmaexist}
Let $t \in \mathbb{Z}.$ If $\lim_{s \to -\infty} \delta(P^{s,t}) = 0,$ then there exists a rank 1 matrix $Q^t$ such that
$$\lim_{s \to -\infty} P^{s,t}(x,y) = Q^t(x,y), \qquad \text{ for all } x,y \in [n].$$
\end{lemma}

\begin{proof}[Proof of Proposition~\ref{deterministicExistenceProp}]
The submultiplicativity applied to  $P^{t_k,t_{k-1}}$ for $k = 1, 2, \dots$ yields $$\lim_{s \to -\infty} \delta(P^{s,t}) = 0.$$ The claim follows from Lemma~\ref{Lemmaexist} with $Q^t(x,\cdot) = \pi_t(\cdot)$ for all $x \in [n]$.
\end{proof}

From now on, we assume that $\pi_t$ exists for all $t\in\mathbb{Z}$ and state its properties.
\begin{lemma}\label{matrixMultProp}
$(\pi_s)_{s \in \mathbb{Z}}$ satisfies
\begin{equation}\label{invarianceProperty}
\pi_t = \pi_rP^{r,t}
\end{equation}
for all $t > r \in \mathbb{Z}.$
\end{lemma}
\begin{proof}
By the Markov property it is evident that
\begin{align*}
\pi_{t+1}(y) &= \lim_{s \to -\infty} \mathbb{P}^P_{x,s}(X_{t+1} = y) 
= \lim_{s \to -\infty} \sum_z \mathbb{P}^P_{x,s}(X_{t+1} = y, X_t = z) \\
&= \lim_{s \to -\infty} \sum_z \mathbb{P}^P_{x,s}(X_{t} = z)P_{t+1}(z,y)
= \sum_z \pi_t(z) P_{t+1}(z,y)
\end{align*}
for every $y \in [n]$ and for every $t \in\mathbb{Z}$, which shows $\pi_{t+1} = \pi_t P^{t,t+1}.$ The claim follows by iteration.
\end{proof}

Lemma~\ref{matrixMultProp} allows to define a ``stationary'' version of $X$: Indeed, to this end it is sufficient to observe that by Lemma~\ref{matrixMultProp}, the sequence of distributions $\mathbb{P}^P_{\pi_t,t}$ of the Markov chain started at time $t$ with initial distribution $\pi_t$, is compatible in the sense that
$$\mathbb{P}^P_{\pi_{t-1},t-1}\big|_{[n]^{\{t, t+1, \dots\}}} = \mathbb{P}^P_{\pi_t,t}.$$
The Kolmogorov extension theorem then implies the existence of the measure $\mathbb{P}^P$ on $[n]^\mathbb{Z}$ under which the chain is ``stationary'', meaning that 
\begin{equation}\label{ppEquation}
\mathbb{P}^P(X_t = y) = \lim_{s \to -\infty} \mathbb{P}^P_{1,s}(X_t = y) = \pi_t(y), \qquad t\in\mathbb{Z}, y\in[n].
\end{equation}
We call $\mathbb{P}^{P}$ the \emph{law of the stationary chain}.

\subsection{Mixing Time} With the measures 
$(\pi_t)_{t \in \mathbb{Z}}$ at our disposal, we define mixing time for time-inhomogeneous Markov chains, as alluded to in (\ref{mixingdef}). Set for $s \in \mathbb{Z}, t \geq 0$, $$d(s,s+t) \coloneqq \sup_x \|P^{s,s+t}(x,\cdot) - \pi_{s+t}(\cdot)\|_\mathrm{TV}.$$

\begin{definition}\label{mixingDefReal}
Let $\varepsilon \in (0,1)$, $s \in \mathbb{Z}$. The $\varepsilon$-mixing time for a sequence $P = (P_t)_{t \in \mathbb{Z}}$ is defined by
$$t_{\mathrm{mix}}^{P}(\varepsilon,s) \coloneqq \inf \{t \geq 0: d(s,s+t) \leq \varepsilon \}.$$
\end{definition}
To make this definition meaningful, we need to confirm that if $d(s, s + t) \le \varepsilon$ for some $t$, then $d(s, s + u) \le \varepsilon$ for all $u \ge t$. Hence whether $d(s,s+\cdot)$ is monotonic is a question that arises naturally. As briefly discussed in the introduction, we can answer this affirmatively. Furthermore, the inequality (\ref{triangle}), with appropriate modifications, also holds in a time-inhomogeneous setting. The following lemma collects these results.
\begin{lemma} \label{lemma3} 
Let $t \geq s \geq u.$ Then
\begin{enumerate}[label=(\alph*)]
\item $d(u,t) \leq d(u,s)$,
\item $d(u,t) \leq d(s,t)$,
\item $d(s,t) \leq \delta(P^{s,t})\leq 2d(s,t)$.
\end{enumerate}
\end{lemma} 
The elementary proof of Lemma~\ref{lemma3} is deferred to Appendix~\ref{appendix:existence}.

\section{Evolving sets for time-inhomogeneous Markov chains}\label{evSetsSection}

In this section we adapt the theory of evolving sets (as introduced by \cite{Morris2003EvolvingSM} for time-homogeneous chains) to time-inhomogeneous Markov chains and apply it to estimate the mixing time given by Definition~\ref{mixingDefReal}. More specifically, we prove an upper bound on the mixing time in the spirit of Corollary 2.3 by Jerrum and Sinclair (\cite{JerrumSinclair}). This corollary can be generalized to non-reversible chains using evolving sets, and a full account of the proof can be found below Theorem 17.10 in~\cite{LevinPeresWilmer2006}. While evolving sets allow for sharper bounds on stronger notions of mixing time (see \cite{Morris2003EvolvingSM}) in the time-homogeneous case, for this paper we content ourselves with generalizing the result presented in~\cite{LevinPeresWilmer2006}.

A crucial observation we have made is that those proofs do not require time-homogeneity to a significant extent - they mainly utilize the existence of a unique stationary distribution~$\pi$ of the chain. We can replace stationarity with our Lemma~\ref{matrixMultProp} to achieve similar results. One important example of that is the martingale property of $(\pi(S_t))_{t \ge 0}$ (see Lemma 17.13. in~\cite{LevinPeresWilmer2006}), where $S_t$ is a stochastic process introduced in (\ref{eqEvoSt}) below (with appropriate modifications for time-homogeneity). In Lemma~\ref{MartProp}, we show that from our definitions and Lemma~\ref{matrixMultProp} it follows that $(\pi_t(S_t))_{t \ge 0}$ is also a martingale.

In the following, we will assume that for every $P_s$, 
\begin{equation}\label{eqLaziness}
P_s(x,x) \geq \frac{1}{2}, \qquad \text{ for all } x \in [n],
\end{equation}
i.e.~that the resulting Markov chain $(X_t)$ is lazy. This avoids all problems with periodicity and is an assumption that is also often made in the homogeneous case. We will point out where it is used in our proofs.

Additionally, we make the following irreducibility assumption: For some $t_0 \in \mathbb{Z}$, assume that for every $x \in [n]$ there exists $s < t_0$ such that $\pi_s(x) > 0$.  Together with (\ref{eqLaziness}), this implies that for every $t \ge t_0$ and for every $x \in [n]$, $\pi_t(x) > 0.$

\subsection{Evolving sets and mixing times}\label{evolvingSetsSection}

Let $P = (P_s)_{s \in\mathbb{Z}}$ be a sequence of transition matrices such that $\pi_t$ exists for every $t \in \mathbb{Z}$, and consider the corresponding Markov chain $X = (X_t)_{t\in\mathbb{Z}}$. The evolving set process $(S_t)$ which we define below is a Markov chain on the space of all subsets of $[n]$. Its time until absorption in either $\emptyset$ or $[n]$ is closely linked to the mixing time of $X$.
In preparation of its construction, recall the law of the stationary chain $\mathbb{P}^P$ from (\ref{ppEquation}) and define for $t\in\mathbb{N}, A \subset [n], y\in[n]$,
\begin{equation}\label{eqEvoQ}
Q_{t+1}(A,y) \coloneqq \mathbb{P}^{P}(X_t \in A, X_{t+1} = y)  = \sum_{x \in A} \pi_t(x) P_{t+1}(x,y)
\end{equation}
and furthermore for $B \subset [n]$
\begin{equation}
Q_{t+1}(A,B) \coloneqq \sum_{y \in B} Q_{t+1}(A,y).
\end{equation}
The quantity $Q_{t+1}(A,B)$ is the ``stationary flow'' from $A$ to $B$ between time $t$ and $t+1$.
Observe that, due to Lemma~\ref{matrixMultProp}, the stationary flow satifies 
\begin{equation}\label{eqFutureQt}
Q_{t+1}([n],y) = \pi_{t+1}(y), \qquad \text {for all } y\in [n],
\end{equation}
and, by (\ref{eqEvoQ}),
\begin{equation}\label{eqPresentQt}
Q_{t+1}(A,[n]) = \sum_{z \in [n]} Q_{t+1}(A,z) = \pi_t(A), \qquad A \subset [n].
\end{equation}
Let $(U_t)_{t \in \mathbb{Z}}$ be a family of i.i.d.~ uniform random variables on $[0,1]$ that is independent of $(X_t)_{t \in \mathbb{Z}}$. For a starting time $t_0 \in \mathbb{Z}$ and a non-empty starting state $S = S_{t_0} \subset [n]$, iteratively set
\begin{equation}\label{eqEvoSt}
S_{t+1} \coloneqq \left\{y \in [n] : \frac{Q_{t+1}(S_t,y)}{\pi_{t+1}(y)} \geq U_{t+1} \right\}, \qquad t \ge t_0.
\end{equation}
This defines a time-inhomogeneous Markov chain $(S_t)_{t \ge t_0}$ on the set of subsets of $[n]$. It is easy to see that $\emptyset$ and $[n]$ are the absorbing states of this chain. 
Let $\mathbb{P}^P_{S,t_0}$ be the distribution of $(S_t)_{t \ge t_0}.$
By (\ref{eqEvoSt}) and the Markov property of $(S_t)_{t \ge t_0}$, we have
\begin{equation}\label{eqEvoCond}
\frac{Q_{t+1}(S_t,y)}{\pi_{t+1}(y)} = \mathbb{P}_{S,t_0}^{P}(y \in S_{t+1} | S_t) = \mathbb{P}_{S_t,t}^{P}(y \in S_{t+1}),
\end{equation}
where $\mathbb{P}_{S_t,t}^{P}$ is the distribution of $(S_s)_{s \ge t}$ started in the state $S_t$.
Combining (\ref{eqEvoQ}) and (\ref{eqEvoCond}) yields
\begin{equation}\label{eqEvo1}
\mathbb{P}_{S,t}^{P}(y \in S_{t+1}) \pi_{t+1}(y) = Q_{t+1}(S,y) = \sum_{z \in S} \pi_t(z) P_{t+1}(z,y).
\end{equation}

We now have all the tools to prove that $(\pi_t(S_t))_{t \ge t_0}$ is a martingale. This will be later used to link the growth of $S_t$ to the bottleneck ratio of the underlying $P_t$. Since the choice of starting time $t_0$ is essentially arbitrary, without loss of generality we assume $t_0 = 0$.

\begin{lemma} \label{MartProp}
The sequence $(\pi_t(S_t))_{t \geq 0}$ is a martingale under $\mathbb{P}_{S,0}^P$ with respect to $(\mathcal{F}_t)_{t \ge 0}$, where $\mathcal{F}_t = \sigma(S_s: 0\le s \le t)$. In particular \begin{equation} \label{MarkovMartProp}
\mathbb{E}^P_{S,t}[\pi_{t+1}(S_{t+1})] = \pi_t(S)
\end{equation}
for every $S \subset [n], t \in \mathbb{N}.$
\end{lemma}
\begin{proof}[Proof of Lemma~\ref{MartProp}] The proof is similar to the time-homogeneous version, Lemma 17.13 in~\cite{LevinPeresWilmer2006}. Since $(S_t)_{t \geq 0}$ is a Markov chain, it suffices to condition on $S_t$ in place of $\mathcal{F}_t$:
\begin{align*}
\mathbb{E}^{P}_{S,0}[\pi_{t+1}(S_{t+1}) | S_t] &= \mathbb{E}^{P}_{S,0}\Bigg[\sum_{z \in [n]} \mathbbm{1}_{\{z \in S_{t+1}\}}\pi_{t+1}(z) \Big| S_t\Bigg] \\
&= \sum_{z \in [n]} \mathbb{P}^{P}_{S,0}(z \in S_{t+1} | S_t) \pi_{t + 1}(z) \stackrel{(\ref{eqEvoCond})}{=} \sum_{z \in [n]}Q_{t+1}(S_t,z) \\ 
&\stackrel{\mathclap{(\ref{eqPresentQt})}}{=} \pi_t(S_t),
\end{align*}
hence it is a martingale. (\ref{MarkovMartProp}) follows immediately. 
\end{proof}

To estimate the mixing time, we express the time-dependent connectivity structure of the underlying state space in terms of the growth of the evolving set. Before we can state the main result of this section, we introduce auxiliary notation. Let
\begin{equation}\label{evoEqgt}
g_t \coloneqq \min_{z\in [n]} \frac{\pi_{t-1}(z)}{\pi_t(z)} \in (0, 1]
\end{equation}
bound the rate of change between $\pi_{t-1}$ and $\pi_t$. In time-homogeneous settings, or when $\pi_t = \pi_{t-1}$ for other reasons, $g_t = 1$. Furthermore denote
\begin{equation}
\pi^\mathrm{min}_t \coloneqq \min_z \pi_t(z),
\end{equation}
which will be frequently used to give quantitative lower bounds on $\pi_t(S)$ for $S \subset [n]$ based on the number of elements in $S$. We further define
\begin{equation}\label{evoEqPsi}
\psi_t(S) \coloneqq 1 - \mathbb{E}^P_{S,t-1}\bigg[\sqrt{\tfrac{\pi_t(S_t)}{\pi_{t-1}(S)}}\bigg]
\end{equation}
and
\begin{equation}\label{evoEqPhi}
\Phi_t(S) \coloneqq \frac{1}{2\pi_{t-1}(S)}(Q_t(S,S^c)+Q_t(S^c,S)).
\end{equation}
Note that $\Phi_t(S)$ normalizes the stationary flow between $S$ and its complement $S^c$ by the size of the set $S$ under $\pi_t$. Finally, we define the \emph{time-dependent bottleneck ratio}
\begin{equation}\label{bottleneckRatio}
\Phi_t^* \coloneqq \inf \{\Phi_t(S) : S \subset [n] \text{ with } \pi_{t-1}(S) \leq 1/2 \}.
\end{equation}
We now have the necessary prerequisites to state the main theorem of this section. 
\begin{theorem} \label{upperBoundMixingTheorem}
Fix $\varepsilon \in (0,1).$
If for some $t > 0$
\begin{equation}\label{upperBoundMixingTheoremEquation}
\frac{1}{\sqrt{\pi^\mathrm{min}_t \pi^\mathrm{min}_0}} \prod_{s = 1}^{t} \left[1 - \frac{1}{2}\left(\frac{1}{2} \frac{g_s}{1-\frac{1}{2}g_s}\Phi^*_s\right)^2\right] \leq 2\varepsilon
\end{equation}
then
$$d(0,t) \leq \varepsilon.$$
\end{theorem}
\begin{remark}
Theorem~\ref{upperBoundMixingTheorem} only yields that $t_\mathrm{mix} (\varepsilon, 0) \le t$ for any $t$ that satisfies (\ref{upperBoundMixingTheoremEquation}), while giving no estimate on the $t$. We refine the statement in Corollary~\ref{MainThCor} and Corollary~\ref{probabilisticMainThCor} to strengthen the link to mixing time.
\end{remark}

\subsection{Proof of Theorem~\ref{upperBoundMixingTheorem}}

The proof of Theorem~\ref{upperBoundMixingTheorem} will be based on the observation that we can relate $\psi_t(S)$ to $d(0,t)$. However, first we show that $\Phi_t(S)$ is closely related to $\psi_t(S)$. We do this in two lemmas below. (In the following we use the standard notation $a \wedge b \coloneqq \min(a,b)$.)
\begin{lemma}\label{lem33}
Let $\varphi_t(S) \coloneqq \frac{1}{2 \pi_{t-1}(S)} \sum_{y \in [n]} (Q_t(S,y) \wedge Q_t(S^c,y))$. Then for every $S \subset [n]$ 
\begin{equation}\label{lem33eq}
1 - \psi_t(S) \leq \frac{\sqrt{1+2\varphi_t(S)} + \sqrt{1-2\varphi_t(S)}}{2} \leq 1 - \frac{\varphi_t(S)^2}{2}.
\end{equation}
\end{lemma}
\begin{proof}
First, note that $\frac{\sqrt{1+2x}+\sqrt{1-2x}}{2} \leq 1-\frac{x^2}{2}$ holds for any real number $x \in [-1/2,1/2].$ Hence, to show the second inequality in (\ref{lem33eq}), it is enough to verify that $\varphi_t(S) \in [-1/2,1/2]$. It is clear that $\varphi_t(S) \geq 0.$ On the other hand
$$\sum_{y \in [n]} (Q_t(S,y)\wedge Q_t(S^c,y)) \leq \sum_{y \in [n]} Q_t(S,y) \stackrel{(\ref{eqEvoQ})}{=} \pi_{t-1}(S)$$
which implies that $\varphi_t(S) \leq 1/2.$ Hence $\varphi_t(S) \in [-1/2, 1/2]$. 

The first inequality in (\ref{lem33eq}) is harder to prove. Recall that $U_t$ denotes the uniform random variable used to generate $S_t$ from $S_{t-1}$. We split the proof by conditioning on $U_t \le 1/2$ first, and $U_t > 1/2$ later. Note that conditioned on $U_t \in [0, 1/2]$, $U_t$ is uniform on~$[0,1/2]$. By (\ref{eqEvoSt}), it is immediate that
\begin{equation}\label{LemproofEq1}
\mathbb{P}^{P}_{S,t-1}(y \in S_t | U_t \le 1/2) = 1 \wedge 2 \frac{Q_t(S,y)}{\pi_t(y)}.
\end{equation}
After multiplying both sides of (\ref{LemproofEq1}) by $\pi_t(y)$, this implies
\begin{align*}
\pi_t(y)\mathbb{P}^{P}_{S,t-1}(y \in S_t | U_t \le 1/2) &= \pi_t(y) \wedge 2Q_t(S,y) \stackrel{(\ref{eqFutureQt})}{=} (Q_t(S,y) + Q_t(S^c,y)) \wedge 2Q_t(S,y) \\
&= Q_t(S,y) + (Q_t(S^c,y) \wedge Q_t(S,y)).
\end{align*}
Summing over all $y \in [n]$ yields
\begin{align*}
\mathbb{E}^{P}_{S,t-1}[\pi_t(S_t) | U_t \le 1/2] %&= \sum_{y \in [n]} \pi_t(y) \mathbb{P}^{t-1}_S(y \in S_t | U_t < 1/2) \\
&= \sum_{y \in [n]} Q_t(S,y) + \sum_{y \in [n]} (Q_t(S^c,y) \wedge Q_t(S,y)) \\
&\stackrel{\mathclap{(\ref{eqPresentQt})}}{=} \pi_{t-1}(S) + 2\pi_{t-1}(S)\varphi_t(S).
\end{align*}
Dividing both sides by $\pi_{t-1}(S)$ and defining $R_t \coloneqq \frac{\pi_t(S_t)}{\pi_{t-1}(S_{t-1})}$ results in 
$$\mathbb{E}^{P}_{S,t-1}[R_t | U_t \le 1/2] = 1 + 2 \varphi_t(S).$$
However, by the martingale property (\ref{MarkovMartProp}),
$$\mathbb{E}^{P}_{S,t-1}[R_t] = \mathbb{E}^P_{S,t-1}[\pi_t(S_t) / \pi_{t-1}(S_{t-1})] = \frac{\pi_{t-1}(S)}{\pi_{t-1}(S)} = 1.$$
Since $\mathbb{P}^{P}_{S,t-1}(U_t \le 1/2) = 1/2$, this implies that
$$\mathbb{E}^{P}_{S,t-1}[R_t | U_t > 1/2] = 1 - 2\varphi_t(S).$$
Note that from (\ref{evoEqPsi}), $1 - \psi_t(S) = \mathbb{E}^P_{S,t-1}[\sqrt{R_t}]$, so by Jensen's inequality we can conclude
\begin{align*}
1 - \psi_t(S)
&= \frac{1}{2}\left(\mathbb{E}^{P}_{S,t-1}[\sqrt{R_t} | U_t \le 1/2] + \mathbb{E}^{P}_{S,t-1}[\sqrt{R_t} | U_t > 1/2]\right) \\
&\leq \frac{1}{2}\left(\sqrt{\mathbb{E}^{P}_{S,t-1}[R_t | U_t \le 1/2]} + \sqrt{\mathbb{E}^{P}_{S,t-1}[R_t | U_t > 1/2]}\right) \\
&= \frac{1}{2}\left(\sqrt{1 + 2\varphi_t(S)} + \sqrt{1 - 2\varphi_t(S)}\right).
\end{align*}
This shows the first inequality in (\ref{lem33eq}) and completes the proof.
\end{proof}

We can now prove a relation between $\psi_t$ and $\Phi_t.$ This result is inspired by Lemma 3 in~\cite{Morris2003EvolvingSM}, with an additional term appearing due to time-inhomogeneity.
\begin{lemma} \label{psiestimate}
For every $t > 0$ and for every set $S \subset [n]$
\begin{equation}\label{evoIneqLemmaPsiPhi}
\psi_t(S) \geq \frac{1}{8} \frac{g_t^2}{(1-\frac{1}{2}g_t)^2}(\Phi_t(S))^2.
\end{equation}
\end{lemma}
\begin{proof}
We show that 
\begin{equation}\label{proofEqStar}
\varphi_t(S) \geq \frac{1}{2} \frac{g_t}{1 - \frac{1}{2}g_t} \Phi_t(S).
\end{equation}
From (\ref{proofEqStar}) and Lemma~\ref{lem33}, (\ref{evoIneqLemmaPsiPhi}) follows.
To prove (\ref{proofEqStar}), consider $y \in S$. Clearly $Q_t(S,y) \geq Q_t(y,y) \geq \frac{1}{2}\pi_{t-1}(y)$, by the laziness (\ref{eqLaziness}), and thus
\begin{equation} \label{LemproofEq2}
\frac{Q_t(S,y)}{\pi_t(y)} \geq \frac{1}{2} \frac{\pi_{t-1}(y)}{\pi_{t}(y)}\geq \frac{1}{2} \min_z \frac{\pi_{t-1}(z)}{\pi_t(z)} \stackrel{(\ref{evoEqgt})}{=} \frac{1}{2}g_t.
\end{equation}
Noting that
\begin{equation*}
Q_t(S^c,y) \stackrel{\mathclap{(\ref{eqFutureQt})}}{=} \pi_t(y) - Q_t(S,y) 
\stackrel{(\ref{LemproofEq2})}\leq \Big(1 - \frac{1}{2}g_t\Big) \pi_t(y)
\end{equation*}
we deduce
\begin{equation} \label{LemproofEq3}
\frac{1}{2}\frac{g_t}{1 - \frac{1}{2}g_t} Q_t(S^c,y) \leq \frac{1}{2}g_t \pi_t(y).
\end{equation}
So, combining (\ref{LemproofEq2}) and (\ref{LemproofEq3}) yields that for $y \in S$
\begin{align*}
Q_t(S,y) \wedge Q_t(S^c,y) &\geq \bigg(\frac{1}{2}g_t \pi_t(y)\bigg) \wedge Q_t(S^c,y) \\
&\geq \bigg(\frac{1}{2}\frac{g_t}{1 - \frac{1}{2}g_t}Q_t(S^c,y)\bigg) \wedge Q_t(S^c,y) \\
&= \frac{1}{2}\frac{g_t}{1 - \frac{1}{2}g_t}Q_t(S^c,y),
\end{align*}
where the last equality holds since $g_t \leq 1$.
On the other hand, if $y \in S^c$, swapping all instances of $S$ and $S^c$, the argument above yields
$$Q_t(S,y) \wedge Q_t(S^c,y) \geq \frac{1}{2}\frac{g_t}{1 - \frac{1}{2}g_t}Q_t(S,y).$$
Summing over all $y \in [n]$, we can separate the case $y \in S$ from $y \in S^c$ to get
\begin{align*}
\sum_{y \in [n]} Q_t(S,y)\wedge Q_t(S^c,y) &= \sum_{y \in S}Q_t(S,y)\wedge Q_t(S^c,y) + \sum_{y \in S^c} Q_t(S,y) \wedge Q_t(S^c,y) \\
&\geq \frac{1}{2} \frac{g_t}{1 - \frac{1}{2}g_t} \sum_{y \in S} Q_t(S^c,y) + \frac{1}{2} \frac{g_t}{1 - \frac{1}{2}g_t} \sum_{y \in S^c} Q_t(S,y) \\
&= \frac{1}{2} \frac{g_t}{1 - \frac{1}{2}g_t}(Q_t(S^{c},S) + Q_t(S,S^c)).
\end{align*}
Multiplying both sides by $\frac{1}{2 \pi_{t-1}(S)}$ results in (\ref{proofEqStar}) which completes the proof.
\end{proof}

In order to prove Theorem~\ref{upperBoundMixingTheorem}, we need the following proposition (which is a statement similar to Lemma 17.12. in~\cite{LevinPeresWilmer2006}) that relates $P^{0,t}$ and $S_t$:
\begin{proposition} \label{Prop2}
For every $t \geq 0$, for every $x,y \in [n]$,
\begin{equation}\label{prop2eq1}
P^{0,t}(x,y) = \frac{\pi_t(y)}{\pi_0(x)} \mathbb{P}_{\{x\},0}^{P}(y \in S_t).
\end{equation}
\end{proposition}
\begin{proof}
We proceed by induction. For $t = 0$ there is nothing to prove. Assume (\ref{prop2eq1}) holds up to $t - 1$ for every $x,y$. Then for arbitrary $x,y \in [n]$
\begin{align*}
P^{0,t}(x,y) &= \sum_{z \in [n]} P^{0,t-1}(x,z)P_t(z,y) = \sum_{z \in [n]} \mathbb{P}^{P}_{\{x\},0}(z \in S_{t-1}) \frac{\pi_{t-1}(z)}{\pi_0(x)}P_t(z,y) \\
&= \frac{\pi_t(y)}{\pi_0(x)} \sum_{z \in [n]} \mathbb{P}_{\{x\},0}^{P}(z \in S_{t-1}) \frac{\pi_{t-1}(z)}{\pi_t(y)}P_t(z,y)
\end{align*}
by using the induction assumption. $\mathbb{E}_{\{x\},0}[\mathbbm{1}_{z \in S_{t-1}}] = \mathbb{P}_{\{x\},0}^{P}(z \in S_{t-1})$ yields that
\begin{align*}
P^{0,t}(x,y) &= \frac{\pi_t(y)}{\pi_0(x)} \mathbb{E}_{\{x\},0}^{P}\bigg[\sum_{z \in S_{t-1}} \pi_{t-1}(z)P_t(z,y) \pi_t(y)^{-1}\bigg] \\
&\stackrel{\mathclap{(\ref{eqEvoQ})}}{=} \frac{\pi_t(y)}{\pi_0(x)} \mathbb{E}_{\{x\},0}^{P}\bigg[\sum_{z \in S_{t-1}} Q_t(z,y)\pi_t(y)^{-1}\bigg] \\
&= \frac{\pi_t(y)}{\pi_0(x)}\mathbb{E}_{\{x\},0}^{P} \left[\pi_t(y)^{-1}Q_t(S_{t-1},y)\right] \stackrel{(\ref{eqEvoCond})}{=} \frac{\pi_t(y)}{\pi_0(x)} \mathbb{E}_{\{x\},0}^{P}\left[\mathbb{P}^{P}_{\{x\},0}(y \in S_t |S_{t-1}) \right] \\
&= \frac{\pi_t(y)}{\pi_0(x)}\mathbb{P}_{\{x\},0}^{P}(y \in S_t).
\end{align*}
This shows the induction step and completes the proof.
\end{proof}
We can now prove the main result of the section.
\begin{proof}[Proof of Theorem~\ref{upperBoundMixingTheorem}]
We follow the strategy of the proof of Theorem 17.10 as outlined in~\cite{LevinPeresWilmer2006}.
For every $t$, we define
\begin{equation}\label{StHash}
S_t^\# \coloneqq
\begin{cases}
S_t, &\mbox{if } \pi_t(S_t) \leq 1/2, \\
S_t^c, &\mbox{if } \pi_t(S_t) > 1/2.
\end{cases}
\end{equation}
It is useful to note in preparation that $(S_t^c)_{t \ge 0}$ is a stochastic process that has the same transition probabilities as $(S_t)_{t \ge 0}$, since
\begin{align*}
S_t^c &= \{y \in [n] : \tfrac{Q_t(S,y)}{\pi_t(y)} < U_t\} \\
&= \{y \in [n] : \tfrac{Q_t(S^c,y)}{\pi_t(y)} \ge 1 - U_t\} \\
&= \{y \in [n] : \tfrac{Q_t(S^c,y)}{\pi_t(y)} \ge \tilde{U_t}\}, 
\end{align*}
where $\tilde{U_t} = 1 - U_t$ is again uniform. This also shows that $S_t^c$, given that $S_{t-1} = S$, behaves like an evolving set started in $S^c$ at time $t-1$, that is to say for $S,T \subset [n]$,
\begin{equation}
\mathbb{P}_{S,t-1}(S_t = T) = \mathbb{P}_{S^c,t-1}(S_t = T^c)
\end{equation}
and in particular for this proof
\begin{equation}\label{complementSetProcess}
\mathbb{E}_{S,t-1}\bigg[\sqrt{\pi_t(S_t^c)}\bigg] = \mathbb{E}_{S^c,t-1}\Big[\sqrt{\pi_t(S_t)}\Big] .
\end{equation}

Recall the notation $R_t = \pi_{t}(S_{t})/\pi_{t-1}(S_{t-1}).$ Observe that
$$\mathbb{E}^{P}_{S,t-1}\big[\sqrt{R_t}\big] \stackrel{(\ref{lem33eq})}{\leq} 1 - \varphi_t(S)^2/2 \stackrel{(\ref{evoIneqLemmaPsiPhi})}{\leq} 1 - \frac{1}{2}\left(\frac{1}{2}\frac{g_t}{1-\frac{1}{2}g_t}\Phi_t(S)\right)^2.$$
Fix $S_{t-1} = S \subset [n]$. If $\pi_{t-1}(S) \leq 1/2$, then
\begin{equation}\label{proofHashIneq1}
\frac{\pi_t(S^\#_t)}{\pi_{t-1}(S_{t-1}^\#)} \leq \frac{\pi_t(S_t)}{\pi_{t-1}(S)}
\end{equation}
because $\pi_t(S_t^\#) \leq \pi_t(S_t)$ by (\ref{StHash}) and $S_{t-1}^\# = S$ by assumption. Taking the expectation of (\ref{proofHashIneq1}) therefore yields
\begin{equation}\label{proofExpIneq1}
\mathbb{E}^{P}_{S,t-1}\Bigg[\sqrt{\tfrac{\pi_t(S^\#_t)}{\pi_{t-1}(S^\#_{t-1})}}\Bigg] \leq \mathbb{E}^{P}_{S,t-1}\big[\sqrt{R_t}\big] \leq 1 - \frac{1}{2}\left(\frac{1}{2}\frac{g_t}{1-\frac{1}{2}g_t}\Phi_t(S)\right)^2.
\end{equation}
Similarly, if $\pi_{t-1}(S) > 1/2$, then
\begin{equation}\label{proofHashIneq2}
\frac{\pi_t(S^\#_t)}{\pi_{t-1}(S^\#_{t-1})} \leq \frac{\pi_t(S_t^c)}{\pi_{t-1}(S^c)}
\end{equation}
because $\pi_t(S_t^\#) \leq \pi_t(S_t^c)$ by (\ref{StHash}) and $S^\#_{t-1} = S^c$ by assumption. Taking the expectation of the square root and starting from $S^c$ at $t-1$, using (\ref{complementSetProcess}), yields
\begin{equation}\label{proofExpIneq2}
\mathbb{E}^P_{S,t-1}\Bigg[\sqrt{\tfrac{\pi_t(S_t^\#)}{\pi_{t-1}(S^\#_{t-1})}}\Bigg] \leq \mathbb{E}^P_{S^c,t-1}\bigg[\sqrt{\tfrac{\pi_t(S_t)}{\pi_{t-1}(S^c)}}\bigg] \leq  1 - \frac{1}{2}\left(\frac{1}{2}\frac{g_t}{1-\frac{1}{2}g_t}\Phi_t(S^c)\right)^2.
\end{equation}

Note that $\pi_{t-1}(S^c) \leq 1/2$ so indeed $\Phi_t(S^c) \geq \inf \{\Phi_t(V) : V \subset [n] \text{ with } \pi_{t-1}(V) \leq 1/2 \}.$ Combining (\ref{proofExpIneq1}) and (\ref{proofExpIneq2}) hence results in
$$\mathbb{E}^{P}_{\{x\},0}\Bigg[\sqrt{\tfrac{\pi_t(S^\#_t)}{\pi_{t-1}(S_{t-1}^\#)}} \Bigg| S_{t-1}\Bigg] \leq 1 - \frac{1}{2}\left(\frac{1}{2} \frac{g_t}{1-\frac{1}{2}g_t}\Phi^*_t\right)^2.$$
Multiplying both sides by $\sqrt{\pi_{t-1}(S^\#_{t-1})}$ and then taking expectations gives by the tower property that
\begin{equation}\label{upperBoundRecursion}
\mathbb{E}^{P}_{\{x\},0}\left[\sqrt{\pi_t(S^\#_t)}\right] \leq \left[1 - \frac{1}{2} \left(\frac{1}{2} \frac{g_t}{1-\frac{1}{2}g_t}\Phi^*_t\right)^2\right] \mathbb{E}^{P}_{\{x\},0}\left[\sqrt{\pi_{t-1}(S_{t-1}^\#)}\right].
\end{equation}
Recursively applying (\ref{upperBoundRecursion}), we arrive at
$$\mathbb{E}^{P}_{\{x\},0}\left[\sqrt{\pi_t(S_t^\#)}\right] \leq \sqrt{\pi_0(x)} \cdot \prod_{s = 1}^{t} \left[1 - \frac{1}{2}\left(\frac{1}{2} \frac{g_s}{1-\frac{1}{2}g_s}\Phi^*_s\right)^2\right].$$ 
Clearly $\sqrt{\pi^{\mathrm{min}}_t} \cdot \mathbb{P}^{P}_{\{x\},0}(S_t^\# \neq \emptyset) \leq \mathbb{E}^{P}_{\{x\},0}\left[\sqrt{\pi_t(S_t^\#)}\right].$ Rearranging the terms gives
\begin{equation}\label{eqThmProof1}
\mathbb{P}^{P}_{\{x\},0}(S_t^\# \neq \emptyset) \leq \frac{\sqrt{\pi_0(x)}}{\sqrt{\pi^{\mathrm{min}}_t}}\cdot \prod_{s = 1}^{t} \left[1 - \frac{1}{2}\left(\frac{1}{2} \frac{g_s}{1-\frac{1}{2}g_s}\Phi^*_s\right)^2\right].
\end{equation}

We will now introduce four identities that relate $d(0,t)$ to (\ref{eqThmProof1}). Let $$\tau \coloneqq \inf\{t\geq 0 : S_t^\# = \emptyset\} = \inf\{t\geq 0 : S_t = \emptyset \text{ or } S_t = [n]\}.$$ By the optional stopping theorem and the law of total expectation
\begin{equation}
\begin{split} 
\pi_0(x) &= \mathbb{E}^P_{\{x\},0}[\pi_{\tau \wedge t}(S_{\tau \wedge t})] \\
&= \mathbb{E}^P_{\{x\},0}\big[\pi_\tau(S_\tau) \big| \tau \leq t\big]\mathbb{P}^P_{\{x\},0}(\tau \leq t) + \mathbb{E}^P_{\{x\},0}\big[\pi_t(S_t) \big| \tau > t\big]\mathbb{P}^P_{\{x\},0}(\tau > t). \label{eqThmProof2}
\end{split}
\end{equation}

On the other hand, by Proposition~\ref{Prop2}, for any $x,y \in [n]$
\begin{align} \label{eqThmProof3}
|P^{0,t}(x,y) - \pi_t(y)| &= \frac{\pi_t(y)}{\pi_0(x)} \Big|\mathbb{P}^{P}_{\{x\},0}(y \in S_t) - \pi_0(x)\Big|.
\end{align}
Note that $\{y \in S_t, \tau \leq t\} \subset \{S_\tau = [n], \tau \leq t\}$ (if $S_\tau = \emptyset$, then $y \notin S_t$ because $\emptyset$ is absorbing), but also
$\{S_\tau = [n], \tau \leq t\} \subset \{y \in S_t, \tau \leq t\}$ ($[n]$ is absorbing, so $S_t = [n]$ as well, thus containing $y$).
Therefore
\begin{equation}
\begin{split} \label{eqThmProof4}
\mathbb{P}^{P}_{\{x\},0}(y \in S_t) &= \mathbb{P}^{P}_{\{x\},0}(y \in S_t, \tau > t) + \mathbb{P}^{P}_{\{x\},0}(y \in S_t, \tau \leq t)  \\
&= \mathbb{P}^{P}_{\{x\},0}(y \in S_t, \tau > t) + \mathbb{P}^{P}_{\{x\},0}(S_\tau = [n], \tau \leq t).
\end{split}
\end{equation}
As a final preparation, note that also
\begin{equation} \label{eqThmProof5}
 \mathbb{P}^P_{\{x\},0}(S_\tau = [n], \tau \leq t) = \mathbb{E}^P_{\{x\},0}\big[\pi_\tau(S_\tau) \big| \tau \leq t\big]\mathbb{P}^P_{\{x\},0}(\tau \leq t).
\end{equation}
Hence, by (\ref{eqThmProof2})--(\ref{eqThmProof5}),
\begin{align*}
|P^{0,t}(x,y) - \pi_t(y)| &= \frac{\pi_t(y)}{\pi_0(x)} \Big|\mathbb{P}^{P}_{\{x\},0}(y \in S_t, \tau > t) - \mathbb{E}^P_{\{x\},0}[\pi_t(S_t) | \tau > t]\mathbb{P}^P_{\{x\},0}(\tau > t)\Big| \\
&\leq \frac{\pi_t(y)}{\pi_0(x)}\mathbb{P}^{P}_{\{x\},0}(\tau > t) = \frac{\pi_t(y)}{\pi_0(x)} \mathbb{P}^{P}_{\{x\},0}(S_t^\# \neq \emptyset).
\end{align*}
It is easy to show that $d(0,t) = \sup_{x} \|P^{0,t}(x,\cdot) - \pi_t(\cdot)\|_\mathrm{TV} \leq \frac{1}{2} \max_{x,y} \frac{|P^{0,t}(x,y) - \pi_t(y)|}{\pi_t(y)}$, see Chapter 4.7 in~\cite{LevinPeresWilmer2006} for a time-homogeneous version. So combining the above with (\ref{eqThmProof1}) results in
\begin{align*}
2d(0,t) &\le \max_{x,y} \frac{|P^{0,t}(x,y) - \pi_t(y)|}{\pi_t(y)} \\
&\leq \max_{x} \frac{1}{\pi_0(x)} \frac{\sqrt{\pi_0(x)}}{\sqrt{\pi^\mathrm{min}_t}}\prod_{s = 1}^{t} \left[1 - \frac{1}{2}\left(\frac{1}{2} \frac{g_s}{1-\frac{1}{2}g_s}\Phi^*_s\right)^2\right] \\
&\leq \frac{1}{\sqrt{\pi^\mathrm{min}_t \pi^\mathrm{min}_0}} \prod_{s = 1}^{t} \left[1 - \frac{1}{2}\left(\frac{1}{2} \frac{g_s}{1-\frac{1}{2}g_s}\Phi^*_s\right)^2\right].
\end{align*}
In order to mix, we want to find $t$, such that $d(0,t) \leq \varepsilon$ is implied. That is clearly the case when the right-hand side is smaller than $2\varepsilon$, which completes the proof.
\end{proof}

It is often impractical to compute all the individual $\Phi_s^*$. In Section 5 we compute a uniform lower bound on $\Phi_s^*$ that does not depend on the time $s$. With this application in mind, the statement of Theorem~\ref{upperBoundMixingTheorem} is more involved than necessary. 
Define
$$\Theta_t \coloneqq \min_{1 \leq s \leq t}\left(\frac{1}{2} \frac{g_s}{1-\frac{1}{2}g_s}\Phi^*_s\right)^2.$$
Note that $\Theta_t = 0$ for every $t \ge s$ if the state space becomes disconnected at time $s \ge 1$. So the following simplification is only useful when studying models where the state space remains connected long enough, otherwise Theorem~\ref{upperBoundMixingTheorem} is preferred. 

\begin{corollary} \label{MainThCor}
Fix $\varepsilon \in (0,1).$ If for some $t > 0$
\begin{equation}\label{corIneq}
t \geq \frac{2}{\Theta_t} \left[\log\left(\frac{1}{2\sqrt{\pi_0^{\mathrm{min}}\pi_t^{\mathrm{min}}}}\right) + \log\left(\varepsilon^{-1}\right)\right],
\end{equation}
then
$$d(0,t) \leq \varepsilon.$$
\end{corollary}
\begin{proof}
By definition of $\Theta_t$,
$$ \prod_{s = 1}^{t} \left[1 - \frac{1}{2}\left(\frac{1}{2} \frac{g_s}{1-\frac{1}{2}g_s}\Phi^*_s\right)^2\right] \leq (1 - \frac{1}{2}\Theta_t)^t \leq e^{-t\Theta_t/2}.$$
If we have
\begin{equation}\label{corProofeq1}
\frac{1}{\sqrt{\pi_t^{\mathrm{min}}\pi_0^{\mathrm{min}}}} e^{-t\Theta_t/2} \leq 2\varepsilon
\end{equation}
then we can apply Theorem~\ref{upperBoundMixingTheorem}. Taking the logarithm of (\ref{corProofeq1}) gives the sufficient condition
$$t \geq \frac{2}{\Theta_t} \left[\log\Big(\tfrac{1}{2\sqrt{\pi_0^{\mathrm{min}}\pi_t^{\mathrm{min}}}}\Big) + \log\left(\varepsilon^{-1}\right)\right]$$ which implies $d(0,t) \le \varepsilon$ and thus
completes the proof.
\end{proof}

Let us write $$F(t) = \frac{2}{\Theta_t} \left[\log\left(\frac{1}{2\sqrt{\pi_0^{\mathrm{min}}\pi_t^{\mathrm{min}}}}\right) + \log\left(\varepsilon^{-1}\right)\right].$$ 
\begin{remark} \label{rhsFt}
In the time-homogeneous case, $F(t) \eqqcolon F$ does not depend on $t$. Thus, one can compute (or estimate) $F$, choose $t = F$, and from Corollary~\ref{MainThCor} conclude that $d(0,F) \le \varepsilon$, i.e.~ $t^{P}_{\mathrm{mix}}(\varepsilon, 0) \le F$. In fact, with only small adjustments, Corollary~\ref{MainThCor} implies Theorem 17.10 in~\cite{LevinPeresWilmer2006} (and as a consequence Corollary 2.3 in~\cite{JerrumSinclair}).

In the time-inhomogeneous case, this is unfortunately not immediate, as we only have $t \ge F(t)$ and $d(0,t) \le \varepsilon$, and apriori cannot conclude $d(0,F(t)) \le \varepsilon$.

One way around this issue is to construct $T(n),\tau(n) \in \mathbb{N}$, such that $T(n) \ge \tau(n)$ and
$$F(t) \le \tau(n) \qquad \text{ for every $t \in \{0, \dots, T(n)\}$.}$$
Then immediately $t^P_\mathrm{mix}(\varepsilon,0) \le \tau(n)$, by simply picking $t = \tau(n)$ and applying Corollary~\ref{MainThCor}. Since the choice of $T(n)$ and $\tau(n)$ depends on the model, we will make this more specific in the following sections after having introduced the random environment.
\end{remark}

\section{Random environment}\label{randEnvSection}
In Sections~\ref{definitionsSection} and~\ref{evSetsSection},  $P = (P_s)_{s\in\mathbb{Z}}$ has been a fixed sequence of transition matrices on $[n]$. In applications we have in mind, this sequence is interpreted as a random sample from an underlying distribution $\mathbf{P}$ on the sequences of transition matrices. We assume that $\mathbf{P}$ is such that the sequence $(P_s)_{s \in \mathbb{Z}}$ of transition matrices is a time-homogeneous Markov chain. For each sample $(P_s)_{s \in \mathbb{Z}}$ we consider the law of the stationary chain $\mathbb{P}^{(P_s)_{s \in \mathbb{Z}}}$ as in (\ref{ppEquation}) and proceed with the tools from Section~\ref{evSetsSection}.

Let $$\mathcal{S}_n \coloneqq \left\{P \in [0,1]^{n^2} : \sum_{y = 1}^{n} P(x,y) = 1 \quad \forall x \in [n] \right\} $$
be the space of all $n \times n$ stochastic matrices on $[n]$. Let $P = (P_s)_{s\in\mathbb{Z}}$ be some time-homogeneous Markov process on $\mathcal{S}_n$ with unique stationary distribution $\Pi$, and write $\mathbf{P}$ for the law of the corresponding stationary Markov process.
In particular, we have
$$\mathbf{P}(P_t \in A) = \Pi(A), \qquad t \in \mathbb{Z}, A \subset \mathcal{S}_n.$$

In this setting, we first give a sufficient condition under which $\pi_t$, as defined in Section~\ref{definitionsSection}, exists $\mathbf{P}$-almost surely. 
\begin{proposition}\label{piExistencePropRandom}
Assume that there exists a measurable set $A \subset \mathcal{S}_n$ with 
$$\varepsilon \coloneqq \sup \{\delta(P) : P \in A\} < 1$$ such that
$$\mathbf{P}\left(\sum_{k=1}^{\infty} \mathbbm{1}_A(P_k) = \infty \right) = 1.$$
Then $\pi_t$ exists $\mathbf{P}$-a.s.\ for every $t \in \mathbb{Z}.$
\end{proposition}
\begin{proof}
Let $t \in \mathbb{Z}$. Since $\mathbf{P}(\sum_{k=1}^\infty \mathbbm{1}_A(P_k) = \infty) = 1$, because of time-homogeneity and stationarity of $P$, for every $m \in \mathbb{N}$ there exists $T_m \in \mathbb{N}$ such that 
$$\mathbf{P}\Bigg(\sum_{k = 1 + t - T_m}^t \mathbbm{1}_A(P_k) \geq m\Bigg) = \mathbf{P}\bigg(\sum_{k=1}^{T_m} \mathbbm{1}_A(P_k) \geq m\bigg) \geq 1 - 2^{-m}.$$
Since $Q \in A$ implies $\delta(Q) \leq \varepsilon$, by submultiplicativity of $\delta$, we also have \\$\{\sum_{k = 1 + t - T_m}^t \mathbbm{1}_A(P_k) \geq m\} \subset \{\delta(P^{t-T_m,t}) \leq \varepsilon^m\} \subset \{ \lim_{s \to -\infty}\delta(P^{s,t}) \leq \varepsilon^m\}$. Hence
$$\mathbf{P}\left(\lim_{s \to -\infty} \delta(P^{s,t}) > \varepsilon^m\right) \leq \mathbf{P}\bigg(\sum_{k = 1 + t - T_m}^t \mathbbm{1}_A(P_k) < m\bigg) \leq 2^{-m}, \qquad m \in \mathbb{N}.$$
By picking $m$ arbitrarily large, $\varepsilon^m \to 0$, we conclude that $\lim_{s \to -\infty} \delta(P^{s,t}) = 0$, $\mathbf{P}$-a.s.
The $\mathbf{P}$-almost sure existence of $\pi_t$ follows for almost every $(P_s)_{s \in \mathbb{Z}}$ fixed separately by Lemma~\ref{Lemmaexist} and the same arguments as in Proposition~\ref{deterministicExistenceProp}.
\end{proof}

In the following, we are interested in the behavior of mixing time as $n \to \infty$. Consider a sequence of probability spaces $(\mathcal{S}_n,\mathcal{A}_n,\mathbf{P}_n)_{n\in\mathbb{N}}$. A sequence of events $(A_n)_{n\in\mathbb{N}}$ (where each $A_n \in \mathcal{A}_n$) is said to occur \emph{with high probability} (w.h.p.)~if
$$\lim_{n \to \infty}\mathbf{P}_n(A_n) = 1.$$
To simplify notation, we write $\mathbf{P}$ in place of $\mathbf{P}_n$. We use this concept of \emph{high probability} in the following corollary, where we give an explicit bound on mixing time that holds w.h.p.~only.

\begin{corollary} \label{probabilisticMainThCor} Assume there exists a constant $\beta > 0$ such that $\pi_t^{\mathrm{min}} \ge n^{-\beta}$ for every $t \in \{0, \dots, n \}$ w.h.p.~and furthermore assume there exists a constant $\kappa > 0$, such that $\Theta_t \ge \kappa$ for every $t\in\{0, \dots, n \}$ w.h.p. Then, for every $\varepsilon \in (0,1)$,
\begin{equation}
\lim_{n \to \infty} \mathbf{P}\bigg(t_\mathrm{mix}(\varepsilon,0) \le 1 + \frac{2}{\kappa} \Big[\log\big(\tfrac{n^\beta}{2}\big)+\log(\varepsilon^{-1})\Big]\bigg) = 1.
\end{equation}
\end{corollary}
\begin{proof}
Fix $\varepsilon \in (0,1)$, let $n \in \mathbb{N}$ be large enough.
For any $t \in \{0, \dots, n\}$, it holds that 
\begin{equation}\label{corIneq1}
\frac{2}{\Theta_t} \Big[\log\big(\tfrac{1}{2\sqrt{\pi_0^{\mathrm{min}}\pi_t^{\mathrm{min}}}}\big)+\log(\varepsilon^{-1})\Big] \le \frac{2}{\kappa} \Big[\log\big(\tfrac{n^\beta}{2}\big)+\log(\varepsilon^{-1})\Big] \le n
\end{equation}
with high probability.
Let $t = \left \lceil{\frac{2}{\kappa} \Big(\log\big(\tfrac{n^\beta}{2}\big)+\log(\varepsilon^{-1})\Big)}\right\rceil \in \{0, \dots, n\}$. By Corollary~\ref{MainThCor}, (\ref{corIneq1}) implies that $d(0,t) \le \varepsilon$ with high probability. Hence
$$\lim_{n \to \infty}  \mathbf{P}\bigg(t_\mathrm{mix}(\varepsilon,0) \le 1 + \frac{2}{\kappa} \Big[\log\big(\tfrac{n^\beta}{2}\big)+\log(\varepsilon^{-1})\Big]\bigg) = 1$$
which completes the proof.
\end{proof}

\begin{remark}
In light of the discussion in Remark $\ref{rhsFt}$, this corollary eliminates the dependency on $t$ of the right-hand side of (\ref{corIneq}). For that purpose, in the notation of Remark~$\ref{rhsFt}$, we chose $T(n) = n$ and $\tau(n) = \frac{2}{\kappa} \big[\log\big(\tfrac{n^\beta}{2}\big)+\log(\varepsilon^{-1})\big]$.
\end{remark}

\begin{remark}
The quantity $t_\mathrm{mix}(\varepsilon,0)$ is a random variable that depends on the \emph{sequence} of transition matrices $(P_s)_{s\in\mathbb{Z}}$. As such, the randomness of the transition matrices and the randomness of the resulting chain are inherently viewed separately.

We point out that even small changes in the setup can significantly alter the interpretation of mixing time: In~\cite{10.1214/17-AAP1289}, the authors consider mixing time as a random variable that depends on the \emph{initial state} of the environment, whereas the dynamics of the environment are observed jointly with the random walk. Therefore the dynamics of the environment can speed up the mixing. As an example, if the environment is likely to undergo significant changes in every time step, the mixing time as defined in~\cite{10.1214/17-AAP1289} is essentially constant.
\end{remark}

\section{Dynamic Erd\H{o}s-R\'enyi graphs}\label{exampleSection}
We now demonstrate how the results from Section~\ref{evSetsSection} and their Corollary~\ref{probabilisticMainThCor} can be applied to a concrete example, a random walk on an Erd\H{o}s-R\'enyi graph that is independently resampled after each time step.

Let $(G_t)_{t\in\mathbb{Z}}$ be a sequence of independent Erd\H{o}s-R\'enyi graphs of size $n \in \mathbb{N}$ with parameter $p \in (0,1).$
We introduce the following notation: We write $\mathrm{deg}_t(x)$ for the (random) degree of vertex $x$ at time $t$ and we denote by $x \sim_t y$ that $x$ and $y$ are connected by an edge in the graph $G_t$. 

Given $G_t$, the transition matrix $P_t$ corresponding to a lazy simple random walk on the graph is given by
\begin{equation*}
P_t(x,y) = 
\begin{cases*}
1/2 & if $x = y$ and $\mathrm{deg}_t(x) \geq 1$, \\
1 & if $x = y$ and $\mathrm{deg}_t(x) = 0$, \\
1/(2\mathrm{deg}_t(x)) & if $x \neq y$ and $x \sim_t y$, \\
0 & otherwise.
\end{cases*}
\end{equation*}
Since the transition matrices are i.i.d., $(P_s)_{s\in \mathbb{Z}}$ is a time-homogeneous Markov chain that has a stationary distribution $\Pi$ determined by the edge probability $p$. As before we write~$\mathbf{P}$ for the distribution of $(P_s)_{s\in\mathbb{Z}}.$

Observe that the complete graph is sampled with positive probability and thus infinitely many $G_t$'s are complete graphs $\mathbf{P}$-almost surely. The assumptions of Proposition~\ref{piExistencePropRandom} are therefore easily verified and it follows that $(\pi_t)_{t \in \mathbb{Z}}$ exists for $\mathbf{P}$-almost every sequence~$(P_t)_{t \in \mathbb{Z}}.$

In the above construction we assume that each graph is independent from the previous. This assumption seems restrictive, however we will see that even under this assumption, the subsequent proofs are not completely trivial. For that reason, we prefer to stick to this simpler setting. One could expand the graph dynamic to a Markovian model where non-existent edges appear with probability $p$ and existing edges disappear with probability~$q$. This introduces correlation between $\pi_t$ and $P_{t+1}$ which makes the proof techniques we present here fail in some crucial aspects.

Additionally, we make some assumptions on the parameter $p$. In particular, we choose~$p$ such that the graphs are strongly connected significantly above the usual connectivity threshold for Erd\H{o}s-R\'enyi graphs. We did not make an effort to  optimize the connectivity requirement, but it is not trivial to significantly lower it.

\begin{theorem} \label{logUpperBound}
Let $p = \frac{\eta \log n}{n-1}$ with $\eta > 50$.
Then there exists a constant $c' = c'(\eta)$, such that for every $\varepsilon \in (0,1)$
\begin{equation}\label{logUpperBoundEq}
\lim_{n \to \infty} \mathbf{P}\left(t_{\mathrm{mix}}(\varepsilon,0) \leq c' [\log n + \log(\varepsilon^{-1})]\right) = 1.
\end{equation}
\end{theorem}

\begin{remark}
By the stationarity of the random graphs, $t_\mathrm{mix}(\varepsilon,0) \stackrel{d}{=} t_\mathrm{mix}(\varepsilon,s)$ for all $s$. Therefore we only consider $t_\mathrm{mix}(\varepsilon,0).$
\end{remark}

\begin{remark}\label{falseIntuition}
Before every time step, all edges in the graph are independently resampled. This leads to the tempting conclusion that the random walk can forget its starting position very quickly and $t_\mathrm{mix}(\varepsilon,0) = c'(\varepsilon)$ for some constant $c'$ independent of $n$. This intuition is \emph{false} (if $p \ll 1$), since our concept of mixing is \emph{quenched}: We fix the transition matrices beforehand and then view mixing time of the random walk given that fixed sequence of transition matrices. We prove a lower bound on mixing time in Theorem~\ref{logLowerBound}, demonstrating that mixing time is indeed not constant.
\end{remark}

To prove Theorem~\ref{logUpperBound}, we apply Corollary~\ref{probabilisticMainThCor}. To this end, we show that w.h.p.~for every $t \in \{0, \dots, n\}$, $\pi_t^{\mathrm{min}} \ge \frac{\alpha_1^*}{n}$ for some $0 < \alpha_1^* < 1$ independent of $n$ and $\Theta_t \geq \kappa > 0$ for some constant $\kappa$ independent of $n$.

\subsection{Lower and upper bounds for $\pi_t$}

In this section, we work under the assumptions of Theorem~\ref{logUpperBound}. That is to say
\begin{equation}\label{pAssumptions}
p = \frac{\eta \log n}{n-1} \text{ with } \eta > 50.
\end{equation}

The aim is to show there exist some constants $\alpha_1^*(\eta) \in (0,1)$, $\alpha_2^*(\eta) > 1$ and $\beta > 0$ independent of $n$, such that for all $n$ large enough
\begin{equation} \label{piconcequation2}
\mathbf{P}\Big( \frac{\alpha_1^*}{n} \le \pi_t(x) \le \frac{\alpha_2^*}{n}, \forall t \in \{0, ..., n\}, \forall x \in [n] \Big) \ge 1 - n^{-\beta}.
\end{equation}

The following lemma provides a first basic, yet useful, upper bound for $\pi_t$.
\begin{lemma}\label{halfBoundLemma}
Let $(P_s)_{s \in \mathbb{Z}}$ be a realization from $\mathbf{P}$, such that $(\pi_s)_{s \in \mathbb{Z}}$ exists. Then $$\pi_t(x) \le \tfrac{1}{2}, \qquad \text{ for all } x \in [n], \text{ for all } t \in \mathbb{Z}.$$
\end{lemma}
\begin{proof}
Note that $P_t(x,x) = \frac{1}{2}$ (which holds unless $x$ is isolated) implies that 
\begin{equation}\label{piless1/2}
\pi_t(x) \stackrel{{(\ref{invarianceProperty})}}{=} \frac{1}{2}\pi_{t-1}(x) + \sum_{y \neq x} \pi_{t-1}(y)P_t(y,x)\leq \frac{1}{2} [\pi_{t-1}(x) + (1 - \pi_{t-1}(x))] \leq \frac{1}{2} 
\end{equation}
for every $t \in \mathbb{Z}, x\in[n]$. If $P_t(x,x) = 1$, then $\pi_t(x) = \pi_{t-1}(x)$. So $\pi_t(x) > \tfrac{1}{2}$ is only possible when $P_s(x,x) = 1$ for all $s < t$, meaning that $x$ has been isolated for the entire history of the graph. Then $\pi_t$ is not well-defined, since $\lim_{s \to -\infty} \mathbb{P}_{x,s}^P(X_t = x) = 1$, yet for any $z \neq x$, $\lim_{s \to -\infty} \mathbb{P}_{z,s}^P(X_t = x) = 0.$ This contradicts the assumption that $(\pi_s)_{s \in \mathbb{Z}}$ exists.
\end{proof}
With the next proposition we gain good control over $\pi_t$. Its proof will take up the remainder of this section.
\begin{proposition}\label{PiConcProp}
There exist constants $\alpha_1^*(\eta) \in(0,1)$, $\alpha_2^*(\eta) > 1$ and $\beta > 0$ independent of $n$, such that for all $n$ large enough there exists
a non-decreasing sequence $(\alpha_1^{(t)})_{0 \le t \le n^2}$ and a non-increasing sequence $(\alpha_2^{(t)})_{0 \le t \le n^2}$ with 
$\alpha_1^{(t)} = \alpha_1^*$ and $\alpha_2^{(t)} = \alpha_2^*$ for all $t \ge n^2 - n$, such that
\begin{equation} \label{piconcequation}
\mathbf{P}\bigg(\frac{\alpha_1^{(t)}}{n} \le \pi_t(x) \le \frac{\alpha_2^{(t)}}{n}, \forall t \in \{0, ..., n^2\}, \forall x \in [n]\bigg) \ge 1 - n^{-\beta}.
\end{equation}
\end{proposition}
\begin{remark} \label{remarkFiveSeven}
By time-invariance of the random graph sequence, (\ref{piconcequation2}) is an immediate consequence of (\ref{piconcequation}), with those same constants $\alpha_1^*(\eta) \in(0,1)$, $\alpha_2^*(\eta) > 1$ and $\beta > 0$. 
\end{remark}

In preparation of the proof of Proposition~\ref{PiConcProp}, we give some definitions and notation. To specify the sequences used in (\ref{piconcequation}), we fix constants $\beta = 0.3, \alpha_1^* = 0.002, \alpha_2^*~=~7, \varepsilon~=~10^{-4}$ independent of $n$. These constants are chosen to make the proof work, but are not optimized to achieve the best possible constant $c'$ in Theorem~\ref{logUpperBound}. We define the sequences
\begin{equation}\label{alphatwo}
\alpha_2^{(t)} = 
\begin{cases}
\frac{n}{2}, &\mbox{if } 0 \le t \le n^{1.1}, \\
(1-\varepsilon)\alpha_2^{(t-1)}, &\mbox{if } n^{1.1} < t \le n^{1.2} \text{ and } (1-\varepsilon)\alpha_2^{(t-1)} \ge \alpha_2^*, \\
\alpha_2^*, &\mbox{otherwise,}
\end{cases}
\end{equation}
and
\begin{equation}\label{alphaone}
\alpha_1^{(t)} = 
\begin{cases}
0, &\mbox{if } 0 \le t \le n^{1.1}, \\
\frac{n}{(16\eta\log n)^{n-1}}, &\mbox{if } n^{1.1} < t \le n^{1.2}, \\
(1 + \varepsilon)\alpha_1^{(t-1)}, &\mbox{if } n^{1.2} < t < n^{2} - n \text{ and } (1+\varepsilon)\alpha_1^{(t-1)} \le \alpha_1^*, \\
\alpha_1^*, &\mbox {otherwise.}
\end{cases}
\end{equation}
Note that $\alpha_1^{(t)}$ and $\alpha_2^{(t)}$ depend on $n$.

We then define the events
\begin{align*}
F_{u,t} &\coloneqq \big\{ \pi_t(x) \leq \tfrac{\alpha_2^{(t)}}{n},\, \forall x \in [n]\big\},\\
F_{l,t} &\coloneqq \big\{ \pi_t(x) \geq \tfrac{\alpha_1^{(t)}}{n},\, \forall x \in [n]\big\},\\
F_t &\coloneqq F_{u,t} \cap F_{l,t},
\end{align*}
so that (\ref{piconcequation}) is equivalent to $\mathbf{P}(\cap_{t \le n^2} F_t) \ge 1 - n^{-\beta}.$
We note that, due to Lemma~\ref{halfBoundLemma},
\begin{equation}\label{EventIsZeroEq}
\mathbf{P}(F_0) = 1.
\end{equation}
Furthermore, we fix $$c_1 = \frac{11}{21}\eta, \quad c_2 = 2\eta.$$ We define
\begin{align*}
D_t &\coloneqq \{ \mathrm{deg}_t(x) \in [c_1 \log n, c_2 \log n],\, \forall x \in [n]\},\\
C_t &\coloneqq \{G_k \text{ is a connected graph for all } 0 \le k < t\}.
\end{align*}
It is well-known for Erd\H{o}s-R\'enyi graphs with $p$ as in (\ref{pAssumptions}) that $D_1$ and $C_1$ are events that occur with high probability. However, Proposition $\ref{PiConcProp}$ requires quantitative estimates which we collect in Lemma~\ref{degreeLemma} and Lemma~\ref{connectedGraphLemma} below. Their proofs are given in Appendix~\ref{appendix:concentration}.

\begin{lemma}\label{degreeLemma}
There exists $\rho > 2 + \beta$, such that for every $n$ large enough and every $t \in \mathbb{Z}$,
\begin{equation}\label{degLemIneq}
\mathbf{P}(D_t^c) \le n^{-\rho}.
\end{equation}
\end{lemma}
\begin{lemma}\label{connectedGraphLemma}
For every $n$ large enough and for every $0 < t \le n^2$
\begin{equation}
\mathbf{P}(C_{t}^c) \le \mathbf{P}(C_{n^2}^c) \le n^{-\frac{\eta}{2}+4}.
\end{equation}
\end{lemma}

Let us now turn our attention to (\ref{piconcequation}). We decompose the complement of the probability in (\ref{piconcequation}), using the notation established above:
\begin{align*}
\mathbf{P}\Bigg(\bigcup_{t = 0}^{n^2}(F_{u,t}^c \cup F_{l,t}^c)\Bigg) &= \mathbf{P}(F_{u,0}^c \cup F_{l,0}^c) + \sum_{t=1}^{n^2} \mathbf{P}(F_0,\dots,F_{t-1},(F_{u,t}^c \cup F_{l,t}^c)) \\
&\stackrel{\mathclap{(\ref{EventIsZeroEq})}}{\le} \sum_{t = 1}^{n^2} (\mathbf{P}(F_0,\dots,F_{t-1},F_{u,t}^c) + \mathbf{P}(F_0,\dots, F_{t-1},F_{l,t}^c)) \\
&\le \sum_{t = 1}^{n^2} (\mathbf{P}(F_{u,t}^c,F_{u,t-1}) + \mathbf{P}(F_{l,t}^c,F_{t-1},C_{t}) + \mathbf{P}(F_{l,t}^c,F_{t-1},C_{t}^c)) \\
&\le n^2\mathbf{P}(C_{n^2}^c) + \sum_{t = 1}^{n^2} (\mathbf{P}(F_{u,t}^c | F_{u,t-1}) + \mathbf{P}(F_{l,t}^c,F_{t-1},C_{t})).
\end{align*}
We have already established in Lemma~\ref{connectedGraphLemma} that $n^2\mathbf{P}(C_{n^2}^c)$ is small, so it remains to bound the final sum.
To that end we will use, for $t \le n^{1.2}$, 
\begin{equation}\label{boundsToProve1}
\mathbf{P}(F_{l,t}^c,F_{t-1},C_{t}) \le \mathbf{P}(F_{l,t}^c,C_{t}) 
\end{equation}
and for $t > n^{1.2}$
\begin{equation}\label{boundsToProve2}
\mathbf{P}(F_{l,t}^c,F_{t-1},C_{t}) \le \mathbf{P}(F_{l,t}^c,F_{t-1}) \le \mathbf{P}(F_{l,t}^c, D_t| F_{t-1}) + \mathbf{P}(D_t^c).
\end{equation}
Further note that 
\begin{equation}\label{boundsToProve3}
\mathbf{P}(F_{u,t}^c|F_{u,t-1}) \le \mathbf{P}(F_{u,t}^c, D_t | F_{u,t-1}) + \mathbf{P}(D_t^c).
\end{equation}
To prove Proposition~\ref{PiConcProp}, it now suffices to prove that the right-hand sides of (\ref{boundsToProve1}), (\ref{boundsToProve2}) and (\ref{boundsToProve3}) are each smaller than $n^{-\rho}$ for some $\rho > 2 + \beta.$ Lemma~\ref{degreeLemma} already bounds $\mathbf{P}(D_t^c) \le n^{-\rho}$ for a (potentially different) $\rho > 2 + \beta.$ 
\begin{remark}\label{largeEnough}
We compute bounds that have a leading term $n^{-\rho}$ for some $\rho > 2 + \beta$ independent of $n$. To simplify and remove all multiplicative constants we reduce $\rho$ very slightly and say the statement holds for all $n$ large enough.
This is not a limitation since Theorem~\ref{logUpperBound} only considers the limit $n \to \infty$.

We will restate explicitly that statements hold for all $n$ large enough only when it is important to clarify which constants do not depend on $n$. The constant $\rho$ varies from statement to statement, but it never depends on $n$.
\end{remark}

The following lemma gives a bound on the first term appearing on the right-hand side of (\ref{boundsToProve3}).

\begin{lemma}\label{descentLemma} There exists $\rho > 2 + \beta$, such that for every $n$ large enough and for every $t \in \{0, \dots, n^2\}$,
\begin{equation}\label{descentLemmaEq}
\mathbf{P}(F_{u,t}^c,D_t | F_{u,t-1}) \le n^{-\rho}.
\end{equation}

\end{lemma}

\begin{proof}[Proof of Lemma~\ref{descentLemma}]
For $t \le n^{1.1}$, $\alpha_2^{(t)} = n/2$, so $\mathbf{P}(F_{u,t}^c | F_{u,t-1}) = 0$ since $\pi_t(x) \le 1/2$ for every $x$, by Lemma~\ref{halfBoundLemma}. Therefore we can reduce the problem to $t > n^{1.1}$. 

If the event $D_t$ occurs, then, as a consequence of Lemma~\ref{matrixMultProp} and by definition of $P_t$, 
\begin{equation}\label{piDtEq}
\pi_t(x) = \frac{1}{2} \pi_{t-1}(x) + \sum_{y \ne x } \pi_{t-1}(y) \frac{\mathbbm{1}_{x\sim_t y}}{2\mathrm{deg}_t(y)} \le \frac{1}{2} \pi_{t-1}(x) + \sum_{y\ne x} \pi_{t-1}(y) \frac{\mathbbm{1}_{x \sim_t y}}{2c_1 \log n}.
\end{equation}
For arbitrary $x \in [n]$,
\begin{align}
\begin{split}\label{descentLemmaKeyEst}
\mathbf{P}(F_{u,t}^c&, D_t | F_{u,t-1}) \le n\mathbf{P}\Big(\pi_t(x) > \frac{\alpha_2^{(t)}}{n}, D_t \Big| F_{u,t-1}\Big) \\
&\stackrel{\mathclap{(\ref{piDtEq})}}{\le} n\mathbf{P}\Big(\frac{1}{2}\pi_{t-1}(x) + \sum_{y \neq x} \frac{1}{2}\pi_{t-1}(y) \frac{\mathbbm{1}_{x \sim_t y}}{c_1 \log n} > \frac{\alpha_2^{(t)}}{n} \Big| F_{u,t-1}\Big) \\
&\le n\mathbf{P}\Big(\frac{1}{2}\frac{n-1}{n}\frac{\alpha_2^{(t-1)}}{\alpha_2^{(t)}} + \sum_{y \neq x} \frac{1}{2} \frac{n-1}{\alpha_2^{(t)}}\pi_{t-1}(y) \frac{\mathbbm{1}_{x \sim_t y}}{c_1 \log n} > \frac{n-1}{n} \Big| F_{u,t-1} \Big) \\
&\le n \mathbf{P}\Bigg(\sum_{y \neq x} \frac{n-1}{\alpha_2^{(t)}} \pi_{t-1}(y) \mathbbm{1}_{x \sim_t y} > \frac{n-1}{n}\bigg(2 - \frac{\alpha_2^{(t-1)}}{\alpha_2^{(t)}}\bigg)c_1 \log n \Bigg| F_{u,t-1}\Bigg),
\end{split}
\end{align}
where the third line follows by multiplying by $n-1$ and dividing by $\alpha_2^{(t)}$, and using that conditioned on the event $F_{u,t-1}$, $\pi_{t-1}(x) \le \alpha_2^{(t-1)}/n$. 

Let $\Lambda_x = \sum_{y \neq x}(n-1) \pi_{t-1}(y) \mathbbm{1}_{x \sim_t y} /\alpha_2^{(t)}$ denote the sum in the last formula. It holds that 
\begin{equation}\label{proofIneqTwo}
\mathbf{E}[\Lambda_x | \pi_{t-1}] = \frac{n-1}{\alpha_2^{(t)}} \frac{\eta \log n}{n-1} (1 - \pi_{t-1}(x)) = \frac{\eta \log n}{\alpha_2^{(t)}}(1 - \pi_{t-1}(x)).
\end{equation}

To proceed, we need a concentration result for $\Lambda_x$ under $\mathbf{P}( \cdot | F_{u,t-1})$, namely for every $\delta > 0$,
\begin{equation}\label{concDeltaEst}
\mathbf{P}\big(\Lambda_x > (1+\delta)\mathbf{E}[\Lambda_x | \pi_{t-1}] \big| F_{u,t-1} \big) \le \exp\Big[{(\delta - (1+\delta)\log (1+\delta))\tfrac{\eta \log n}{2\alpha_2^{(t)}}}\Big].
\end{equation}
We postpone the proof of (\ref{concDeltaEst}) and first complete the proof of (\ref{descentLemmaEq}). We set
\begin{equation*}
\hat{\delta} \coloneqq
\begin{cases}
\frac{11}{21}(1 - 3\varepsilon)\alpha_2^{(t-1)} - 1, &\qquad\mbox{if } n^{1.1} < t \le n^{1.2} \text{ and } \alpha_2^{(t)} \ne \alpha_2^*, \\
\frac{77}{21}(1 - 3\varepsilon) - 1, &\qquad\mbox{if } n^{1.2} < t \le n^{2} \text{ or } \alpha_2^{(t)} = \alpha_2^*.
\end{cases}
\end{equation*}
Then it is easily shown by recalling $\beta = 0.3$, $c_1 = \frac{11}{21}\eta$, $\alpha_2^{*} = 7$, $\eta > 50$, $\varepsilon = 10^{-4}$, that
\begin{equation}\label{deltaCondi1}
\frac{n-1}{n}\bigg(2 - \frac{\alpha_2^{(t-1)}}{\alpha_2^{(t)}}\bigg) c_1 \log n \ge (1 + \hat{\delta})\mathbf{E}[\Lambda_x | \pi_{t-1}].
\end{equation}
Moreover there exists $\rho > 2 + \beta$ independent of $\hat{\delta}, t, n$, such that
\begin{equation}\label{deltaCondi2}
\exp\Big[{(\hat{\delta} - (1+\hat{\delta})\log (1+\hat{\delta}))\tfrac{\eta \log n}{2\alpha_2^{(t)}}}\Big] \le n^{-\rho -1}.
\end{equation}
By (\ref{descentLemmaKeyEst}) and (\ref{concDeltaEst}) it follows that there exists a $\rho > 2 + \beta$ such that for every $n$ large enough,
$$\mathbf{P}(F_{u,t}^c, D_t | F_{u,t-1}) \le n^{-\rho}.$$ 
Since this holds for every $n^{1.1} < t \le n^2$, (\ref{descentLemmaEq}) follows.

So let us now turn our attention to (\ref{concDeltaEst}). The inequality resembles a standard result about weighted sums of Bernoulli random variables achieved by the means of Chernoff bounds, and some variant and discussion of it can be found for example as Theorem 1 in~\cite{RAGHAVAN1988130}. However, the expression (\ref{concDeltaEst}) has a key difference to classical theory, namely that $\mathbf{E}[\Lambda_x | \pi_{t-1}]$ and the weights $(n-1)\pi_{t-1}(y)/\alpha_2^{(t)}$ are random variables themselves. If a weight can randomly be large while all other weights are comparatively small, the sum $\Lambda_x$ is mainly affected by the outcome of a single Bernoulli random variable, and thus does not concentrate.

However, as it turns out, the condition on the event $F_{u,t-1}$ is the deciding ingredient. When $t$ is small and $F_{u, t-1}$ is barely restricting the weights in size, by construction of $\alpha_2^{(t-1)}$, the right-hand side of (\ref{concDeltaEst}) is close to 1. When $t$ is large and the right-hand side tightens, the condition $F_{u,t-1}$ limits how large single weights can be, which improves concentration. In essence, with increasing $t$, (\ref{concDeltaEst}) improves itself accordingly..

To show (\ref{concDeltaEst}), we use the Markov inequality for conditional probability: For any event $F$ with positive probability, and random variables $X \ge 0$, $Y > 0$,
\begin{equation}\label{MarkovianInequality}
\mathbf{P}(X \ge Y | F) \le \frac{\mathbf{E}[\mathbbm{1}_F X/Y]}{\mathbf{P}(F)}.
\end{equation}
For any $\lambda > 0$, with $X = e^{\lambda \Lambda_x}$, $Y = \exp\big[\lambda(1 + \delta)\eta \log n(1 - \pi_{t-1}(x))/\alpha_2^{(t)}\big]$, $F = F_{u,t-1}$, the expectation on the right-hand side of (\ref{MarkovianInequality}) becomes
\begin{equation}\label{proofIneqThree}
\mathbf{E}\Big[e^{\lambda \Lambda_x} e^{-\eta\log n(1 - \pi_{t-1}(x))(1+\delta)\lambda/\alpha_2^{(t)}}\mathbbm{1}_{F_{u,t-1}} \Big].
\end{equation}
We emphasize here that $\pi_{t-1}(x)$ is random, however it is independent from all the Bernoulli random variables $\mathbbm{1}_{x \sim_t y}$ in $\Lambda_x$. In fact, the $\mathbbm{1}_{x \sim_t y}$ are also independent from $\mathbbm{1}_{F_{u,t-1}}$. The tower property yields
$$(\ref{proofIneqThree}) = \mathbf{E}\Bigg[e^{-\eta\log n(1 - \pi_{t-1}(x))(1+\delta)\lambda/\alpha_2^{(t)}}\mathbbm{1}_{F_{u,t-1}} \mathbf{E}\bigg[e^{\lambda \sum_{y \neq x}(n-1) \pi_{t-1}(y) \mathbbm{1}_{x \sim_t y} / \alpha_2^{(t)}} \bigg| \pi_{t-1}\bigg]\Bigg]$$
since $\mathbbm{1}_{F_{u,t-1}}$ is also $\sigma(\pi_{t-1})$-measurable. Exploiting the independence of the $\mathbbm{1}_{x \sim_t y}$'s
\begin{align*}
\mathbf{E}\bigg[e^{\lambda \sum_{y \neq x}(n-1) \pi_{t-1}(y) \mathbbm{1}_{x \sim_t y}/\alpha_2^{(t)}} \bigg| \pi_{t-1}\bigg] &= \prod_{y \neq x} \mathbf{E}\bigg[e^{\lambda (n-1)\pi_{t-1}(y) \mathbbm{1}_{x \sim_t y}/\alpha_2^{(t)}} \bigg| \pi_{t-1}\bigg] \\
&= \prod_{y \neq x} \Big(1 - p + pe^{\lambda (n-1) \pi_{t-1}(y)/\alpha_2^{(t)}}\Big),
\end{align*}
where the equalities hold $\mathbf{P}$-almost surely.
Let $Z_y \coloneqq (n-1)\pi_{t-1}(y)/\alpha_2^{(t)}.$
Using the inequality $(1+z) \le e^z$ and choosing $\lambda = \log(1 + \delta)$, the right-hand side is bounded by
\begin{align*}
\prod_{y \neq x} \Big(1 - p + pe^{\log(1+\delta) Z_y}\Big) &\le \prod_{y \neq x} \exp[p((1+\delta)^{Z_y} - 1)].
\end{align*}
On the event $F_{u,t-1}$, $Z_y \in (0,1].$ We can hence apply the inequality $(1+z)^l \le 1+lz$ for $l \in (0,1], z > 0$ to arrive at
\begin{equation}
\mathbbm{1}_{F_{u,t-1}} \prod_{y \neq x} \exp[p((1+\delta)^{Z_y} - 1)] \le \mathbbm{1}_{F_{u,t-1}} \prod_{y \neq x} \exp[p\delta Z_y].
\end{equation}
Note that $$\prod_{y \neq x} \exp[p\delta Z_y] = \exp[p\delta(n-1)(1 - \pi_{t-1}(x))/\alpha_2^{(t)}].$$
Using that $p = \frac{\eta \log n}{n-1}$ and  $1 - \pi_{t-1}(x) \ge \frac{1}{2}$, the calculation above implies that for $\lambda = \log(1+\delta)$
\begin{align}
\begin{split}
(\ref{proofIneqThree}) &\le \mathbf{P}(F_{u,t-1}) \Bigg[\frac{e^\delta}{(1+\delta)^{1+\delta}}\Bigg]^{\frac{\eta \log n}{2\alpha_2^{(t)}}} \\
&= \mathbf{P}(F_{u,t-1}) \exp\Big[{(\delta - (1+\delta)\log (1+\delta))\tfrac{\eta \log n}{2\alpha_2^{(t)}}}\Big].
\end{split}
\end{align}
This, together with (\ref{MarkovianInequality}), implies (\ref{concDeltaEst}).
\end{proof}
We can now show that the gradually ascending lower bound $F_{l,t}$ also holds with high probability.
\begin{lemma}\label{ascentLemmaHelper}
There exists $\rho > 2 + \beta$ such that
for every $n$ large enough and for every $n^{1.2} < t \le n^2$,
\begin{equation}\label{goodGraphCondition}
\mathbf{P}(F_{l,t}^c, F_{t-1}) \le n^{-\rho},
\end{equation}
and for every $0 \le t \le n^{1.2}$,
\begin{equation}\label{connectedGraphCondition}
\mathbf{P}(F_{l,t}^c, C_{t}) \le n^{-\rho}.
\end{equation}
\end{lemma}
\begin{proof}
We first show (\ref{goodGraphCondition}). Its proof uses similar ideas as the proof of Lemma~\ref{descentLemma}. For $t > n^{1.2}$, intersecting $F_{l,t}^c, F_{t-1}$ with $D_t$ and $D_t^c$ yields
\begin{equation*} 
\mathbf{P}(F_{l,t}^c, F_{t-1}) \le \mathbf{P}(F_{l,t}^c, D_t | F_{t-1}) + \mathbf{P}(D_t^{c}),
\end{equation*}
and we know from Lemma~\ref{degreeLemma} that $\mathbf{P}(D_t^{c})$ is small. So it remains to consider $\mathbf{P}(F_{l,t}^c, D_t | F_{t-1}).$ We fix an arbitrary $x \in [n]$. Bounding $\pi_t(x) \ge \frac{1}{2}\pi_{t-1}(x) + \sum_{y \neq x} \pi_{t-1}(y) \frac{\mathbbm{1}_{x \sim_t y}}{2 c_2 \log n}$ and repeating the steps in Lemma~\ref{descentLemma}, we arrive at
\begin{align}\label{gammaEqBigSystem}
\begin{split}
\mathbf{P}(F_{l,t}^c&,D_t|F_{t-1}) \le n\mathbf{P}\Big(\frac{1}{2}\pi_{t-1}(x) + \sum_{y \neq x} \frac{1}{2}\pi_{t-1}(y) \frac{\mathbbm{1}_{x \sim_t y}}{2 c_2 \log n} < \frac{\alpha_1^{(t)}}{n} \Big| F_{t-1}\Big) \\
&\le n\mathbf{P}\Big(\frac{\alpha_1^{(t-1)}}{n} \frac{n-1}{\alpha_2^{(t)}} + \frac{1}{c_2 \log n} \sum_{y \neq x} \frac{n-1}{\alpha_2^{(t)}}\pi_{t - 1}(y) \mathbbm{1}_{x \sim_t y} < 2 \frac{n-1}{n} \frac{\alpha_1^{(t)}}{\alpha_2^{(t)}}\Big| F_{t-1}\Big) \\
&\le n\mathbf{P}\Big(\sum_{y \neq x} \frac{n-1}{\alpha_2^{(t)}} \pi_{t-1}(y) \mathbbm{1}_{x \sim_t y}< \frac{n-1}{n}\frac{2\alpha_1^{(t)}-\alpha_1^{(t-1)}}{\alpha_2^*} c_2 \log n \Big| F_{t-1}\Big)
\end{split}
\end{align}
where the second line follows by multiplying with $(n-1)/\alpha_2^{(t)}$. Since $t > n^{1.2}$ and hence $\alpha_2^{(t)} = \alpha_2^*$, we know conditional on the event $F_{t-1}$ that $(n-1)\pi_{t-1}(y)/\alpha_2^* \le 1.$ We write the left-hand side inside the last probability in (\ref{gammaEqBigSystem}) as $M_x = \sum_{y \neq x} (n-1)\pi_{t-1}(y)\mathbbm{1}_{x \sim_t y}/\alpha_2^*$. It holds that $$\mathbf{E}[M_x|\pi_{t-1}] = \eta\log n(1 - \pi_{t-1}(x))/\alpha_2^*.$$

For all $\gamma \in (0,1)$ we will show the concentration result 
\begin{equation}\label{gammaConcRes}
\mathbf{P}\big(M_x < (1- \gamma)\mathbf{E}[M_x | \pi_{t-1}] \big| F_{t-1}\big) \le \exp\Big(-\frac{\gamma^2}{2}\frac{\eta\log n}{\alpha_2^*}(1 - \tfrac{\alpha_2^*}{n})\Big).
\end{equation}
We postpone the proof of (\ref{gammaConcRes}) and first complete the proof of (\ref{goodGraphCondition}).
For $n^{1.2} < t \le n^{2}$, we fix
\begin{equation*}
\hat{\gamma} \coloneqq
\begin{cases}
1 - 4(1+2\varepsilon)\alpha_1^{(t-1)}, &\quad\mbox{if } n^{1.2} < t \le n^{2} - n \text{ and } \alpha_1^{(t)} \ne \alpha_1^*, \\
1 - 4\alpha_1^*, &\quad\mbox{if } n^{2} - n < t \le n^{2} \text{ or } \alpha_1^{(t)} = \alpha_1^*.
\end{cases}
\end{equation*}
If we are in the situation that $\alpha_1^{(t)} \ne \alpha_1^*$, then indeed $\alpha_1^{(t)} = (1 + \varepsilon)\alpha_1^{(t-1)}$ by definition, hence
$$2\alpha_1^{(t)} - \alpha_1^{(t-1)} = (1 + 2\varepsilon)\alpha_1^{(t-1)}$$
By recalling the values of $\varepsilon = 10^{-4}$ and $c_2 = 2\eta$, we see immediately that $\hat{\gamma} \in (0,1)$ and furthermore
\begin{equation}\label{gammaEquation}
(1 - \hat{\gamma}) \ge \frac{n-1}{n}\Big(2 \alpha_1^{(t)} - \alpha_1^{(t-1)}\Big)\frac{2c_2}{\eta}.
\end{equation}
In the case $\alpha_1^{(t)} = \alpha_1^*$, it is easy to see that (\ref{gammaEquation}) also holds. Furthermore, since $\alpha_2^* = 7$, it is a short computation to show that for $\rho = 2.4 > 2 + \beta$,
\begin{equation}\label{gammaConcRes2}
\exp\Big(-\frac{\hat{\gamma}^2}{2}\frac{\eta\log n}{\alpha_2^*}\big(1 - \frac{\alpha_2^*}{n}\big)\Big) \le n^{-\rho - 1}.
\end{equation}
By (\ref{gammaEqBigSystem})--(\ref{gammaConcRes2}), we conclude that
\begin{equation}
\mathbf{P}(F_{l,t}^c, D_t | F_{t-1}) \le n^{-\rho},
\end{equation}
and hence (\ref{goodGraphCondition}).

For the case $n^{1.2} < t \le n^2$ it remains to show the concentration result (\ref{gammaConcRes}). Similar to Lemma~\ref{descentLemma}, we show an upper bound for
\begin{equation}\label{expectationToBound}
\mathbf{E}\Big[e^{-\gamma M_x} e^{-\eta\log n(1 - \pi_{t-1}(x))(1-\gamma)\gamma/\alpha_2^*}\mathbbm{1}_{F_{t-1}} \Big].
\end{equation}
By the tower property,
\begin{equation*}
(\ref{expectationToBound}) = \mathbf{E}\Big[e^{-\eta\log n(1 - \pi_{t-1}(x))(1-\gamma)\gamma/\alpha_2^*}\mathbbm{1}_{F_{t-1}} \mathbf{E}\Big[e^{-\gamma \sum_{y \neq x} (n-1)\pi_{t-1}(y)\mathbbm{1}_{x \sim_t y}/\alpha_2^*}  \Big| \pi_{t-1}\Big] \Big].
\end{equation*}
We now consider the inner conditional expectation. We write $$Z_y \coloneqq (n-1)\pi_{t-1}(y)/\alpha_2^*$$ and it holds that
\begin{align*}
\mathbf{E}\Big[e^{-\gamma \sum_{y \neq x} Z_y\mathbbm{1}_{x \sim_t y}}  \Big| \pi_{t-1}\Big] &= \prod_{y \neq x} \mathbf{E}\Big[e^{-\gamma Z_y\mathbbm{1}_{x \sim_t y}}  \Big| \pi_{t-1}\Big] \\
&\le \prod_{y \neq x} (1 - p + pe^{-\gamma Z_y}) \le \exp\Big[{\sum_{y \neq x} - p + pe^{-\gamma Z_y}}\Big].
\end{align*}
Since $0 < Z_y^{2} \le Z_y \le 1$ on the event $F_{t-1}$, and $-1 + e^{-z} \le - z + z^2/2$ for $z > 0$, we obtain
\begin{equation}
\mathbbm{1}_{F_{t-1}}\exp\Big[{\sum_{y \neq x} - p + pe^{-\gamma Z_y}}\Big]\le \mathbbm{1}_{F_{t-1}}e^{\sum_{y \neq x} p(-\gamma Z_y + \gamma^2 Z_y/2)}.
\end{equation}
Noting that $\sum_{y \neq x} Z_y = (1 - \pi_{t-1}(x))\eta \log n / \alpha_2^*$ then finally yields
\begin{equation*}
(\ref{expectationToBound}) \le  \mathbf{P}(F_{t-1}) e^{-\gamma^2 \eta \log n (1 - \pi_{t-1}(x))/(2\alpha_2^*)} \le \mathbf{P}(F_{t-1})e^{-\gamma^2 \eta \log n (1 - \alpha_2^*/n)/(2\alpha_2^*)}
\end{equation*}
since $\pi_{t-1}(x) \le \alpha_2^*/n$ on the event $F_{t-1}$. Applying (\ref{MarkovianInequality}) completes the proof of (\ref{gammaConcRes}).

We now turn to (\ref{connectedGraphCondition}). For $t \le n^{1.1}$, by (\ref{alphaone}),  $\alpha_1^{(t)} = 0$, which makes $F_{l,t}^c$ trivially empty, so $\mathbf{P}(F_{l,t}^c) = 0$. 
For $n^{1.1} < t \le n^{1,2}$, let $O_k = \{\mathrm{deg}_k(x) \le 8\eta \log n, \, \forall x \in [n]\}.$ With the same argument as for the upper bound in the proof of Lemma~\ref{degreeLemma} (see Appendix~\ref{appendix:concentration}), we see that $\mathbf{P}(O_k^{c}) \le n^{-\rho-2}$ for some $\rho > 2 + \beta$.
Therefore
\begin{align*}
\mathbf{P}(F_{l,t}^c, C_{t}) &= \mathbf{P}\Big(F_{l,t}^c \Big| C_{t}, \bigcap_{k = 0}^{t-1} O_k\Big)\mathbf{P}\Big(C_{t}, \bigcap_{k = 0}^{t-1} O_k\Big) + \mathbf{P}\Big(F_{l,t}^c, C_{t}, \bigcup_{k = 0}^{t-1} O_k^c\Big) \\
&\le \mathbf{P}\Big(F_{l,t}^c \Big| C_{t}, \bigcap_{k = 0}^{t-1} O_k\Big) + n^{-\rho}.
\end{align*}
We now just consider $$\mathbf{P}\Big(F_{l,t}^c \Big| C_{t}, \bigcap_{k = 0}^{t-1} O_k\Big).$$
For $t > n^{1.1}$, the condition $$C_{t} \cap \Big(\bigcap_{k = 0}^{t-1} O_k\Big)$$ implies that the graph has been connected for at least the last $n^{1.1}$ time steps, with all degrees bounded by $8\eta\log n$. In such a graph, if $x \sim_s y$, then $P^{s,s+1}(x,y) \ge \frac{1}{16\eta \log n}.$ 

We say that there is a path from $x$ to $y$ of length $k$, starting at time $s$, if there exist $x_1, \dots, x_{k-1} \in [n]$, such that $x \sim_{s+1} x_1$, $x_1 \sim_{s+2} x_2, \dots, x_{k-1} \sim_{s+k} y$ and we write this path as
$$[x,y]_{s}^{s+k} = (x, x_1, \dots, x_{k-1}, y)_s^{s+k}.$$
We allow vertices on the path to be equal, that is to say $x_i = x_j$ for some $i \ne j$.

If there exists a path $[x,y]_s^{s+n-1}$, then
\begin{equation}\label{conEq1}
P^{s,s+n-1}(x,y) \ge \frac{1}{(16\eta\log n)^{n-1}}.
\end{equation}
If we can show that (\ref{conEq1}) holds for all $x,y$ on $C_{t} \cap \big(\bigcap_{k = 0}^{t-1} O_k\big)$, then 
$$\pi_t(x) = \sum_{y}\pi_{t-n+1}(y)P^{t-n+1,t}(y,x) \ge \frac{1}{(16\eta \log n)^{n-1}} \stackrel{(\ref{alphaone})}{=} \frac{\alpha_1^{(t)}}{n}$$ yields that $$\mathbf{P}\Big(F_{l,t}^c \Big| C_{t}, \bigcap_{k = 0}^{t-1} O_k\Big) = 0.$$
So it remains to show that there is indeed a path from $x$ to $y$ for any two vertices $x,y \in [n]$. Consider the following set process (for simplicity we start the process at time $k = 0$, but everything can be shifted accordingly): Let $T^{0}_x \coloneqq \{x\}$,
$$T^{k}_x \coloneqq \{v\in [n] : \exists z \in T_x^{k-1} \text{ s.t. } v \sim_{k} z\} \cup T^{k-1}_x$$
that is to say given a set of vertices $T^{k-1}_x$, $T^{k}_x$ is the set of all vertices that are neighbors of vertices in $T^{k-1}_x$ at time $k$.
This makes $T^{k}_x$ the set of all vertices that are reachable from $x$ in $k$ steps.
Since the graph is connected, there exists one edge that connects $T^{k}_x$ with its complement $(T^{k}_x)^c$ at time $k+1$ if the complement is not empty. In particular, this either implies $|T_x^{k+1}| > |T_x^{k}|$ or the complement is the empty set, that is to say $T^{k}_x = [n].$ Thus $T^{n-1}_x = [n]$, or equivalently, after $n-1$ steps the entire graph is reachable. This proves the existence of a path from $x$ to $y$ of length $n-1$ and completes the proof.
\end{proof}

The proof of Proposition~\ref{PiConcProp} is now a mere application of the lemmas above.

\begin{proof}[Proof of Proposition~\ref{PiConcProp}]
Inserting the Lemmas~\ref{degreeLemma}--\ref{ascentLemmaHelper} into the inequalities (\ref{boundsToProve1})--(\ref{boundsToProve3}) shows equation (\ref{piconcequation}).
\end{proof}

\subsection{Lower bound on $\Theta_t$}
In this section, we show that there exists a constant $\kappa > 0$, independent of $n$, such that $\Theta_t \ge \kappa$ for every $t \in \{0, \dots, n\}$ with high probability. This is the final step in the proof of Theorem~\ref{logUpperBound}, which is then immediate from Corollary~\ref{probabilisticMainThCor}. We recall $c_1 = \frac{11}{21}\eta, c_2 = 2\eta, \alpha_1^* = 0.002, \alpha_2^* = 7.$
\begin{proposition}\label{thetaBoundProp}
It holds that
\begin{equation}\label{thetaBound}
\lim_{n \to \infty} \mathbf{P}\Bigg(\Theta_t \geq \bigg(\frac{\alpha_1^*}{2\alpha_2^*-\alpha_1^*}\cdot \frac{\alpha_1^*}{32(\alpha_2^*)^2}\bigg)^2, \forall t \in \{0, \dots, n\}\Bigg) = 1.
\end{equation}
\end{proposition}
\begin{proof}The proof adapts a strategy presented in Chapter 6 of \cite{durrett_2006}. in~\cite{durrett_2006} only the case of static graphs is considered, but our previous bounds on the degrees and on $\pi_t$ are strong enough to transfer the strategy to a dynamic situation.

In this proof, we bound $\Phi_t^*$ and $g_t$ from below. Let $$\mathcal{E} = \Big\{\frac{\alpha_1^*}{n} \le \pi_{t}(x) \le \frac{\alpha_2^*}{n}, \forall t\in\{0,\dots,n\}, \forall x \in [n]\Big\}$$ and recall from Proposition~\ref{PiConcProp} that $\mathbf{P}(\mathcal{E}) \ge 1 - n^{-\beta}$ for some constant $\beta > 0$. Observe that for any event $A$ that is deterministic on $\mathcal{E}$, that is to say $\mathbf{P}(A | \mathcal{E}) = 1$, we also have
\begin{align}\label{deterministicOnE}
\mathbf{P}(A) \ge \mathbf{P}(A | \mathcal{E}) \mathbf{P}(\mathcal{E}) \ge 1 - n^{-\beta}.
\end{align}

By definition for every set $S \subset [n]$, $Q_t(S^c,S) \geq 0$, so on the set $\mathcal{E}$, we have
\begin{equation}\label{firstPhitEstimate}
\Phi_t(S) \geq \frac{Q_t(S,S^c)}{2\pi_{t-1}(S)} \geq \frac{Q_t(S,S^c)}{2\alpha_2^*|S|}n.
\end{equation}
In particular (\ref{firstPhitEstimate}) holds with probability $1 - n^{-\beta}$ for all $t \in \{1, \dots, n\}$.

We now find a lower bound for $Q_t(S,S^c)$.
Let $e_t(x,S^c) = \sum_{y \in S^c} \mathbbm{1}_{y \sim_t x} $ be the number of edges from $x$ into $S^c$ at time $t$, and $e_t(S,S^c) = \sum_{x \in S} e_t(x,S^c).$ 
In analogy to (\ref{deterministicOnE}), 
$$\mathbf{P}(\cap_{t = 1}^{n} D_t \cap \mathcal{E}) \ge 1 - 2n^{-\beta},$$
and therefore an event that is deterministic on $\cap_{t = 1}^{n} D_t \cap \mathcal{E}$ will also have a probability greater than $1 - 2n^{-\beta}$. We henceforth write $\forall t$ meaning for all $t \in \{1, \dots, n\}$.
For every $S \subset [n]$ and $\forall t$, it is deterministic on $\cap_{t = 1}^{n} D_t \cap \mathcal{E}$ that
\begin{align*}
Q_t(S,S^c) &= \sum_{x \in S} \sum_{y \in S^c} \pi_{t-1}(x)P_t(x,y) = \sum_{x \in S} \pi_{t-1}(x) \sum_{\substack{y \in S^c \\ y \sim_t x}} \frac{1}{2 \mathrm{deg}_t(x)} \\
&\geq \sum_{x \in S} \pi_{t-1}(x) \frac{1}{2c_2\log n}e_t(x,S^c) \\
&\geq \frac{1}{2c_2\log n} \frac{\alpha_1^*}{n} \sum_{x \in S} e_t(x,S^c) %\\
= \frac{1}{2c_2\log n} \frac{\alpha_1^*}{n} e_t(S,S^c) \eqqcolon m_t(S).
\end{align*}
So in particular $\mathbf{P}(Q_t(S,S^c) \ge m_t(S), \forall S \subset [n], \forall t) \ge 1 - 2n^{-\beta}.$

Let us now consider $m_t(S)$. To find a lower bound, we bound $e_t(S,S^c)$. 
By definition of $\Phi_t^*$ we need only consider $S \subset [n]$ with $\pi_{t-1}(S) \le 1/2$.  We claim that, on $\mathcal{E}$, $\pi_{t-1}(S) \le 1/2$ implies $|S| \le \big(1 - \frac{1}{2\alpha_2^*}\big)n$. To prove this claim, it suffices to observe that on $\mathcal{E}$, for any $S \subset [n]$, 
$$\pi_{t-1}(S) \le |S|\frac{\alpha_2^*}{n}.$$
If $|S^c| < \frac{n}{2\alpha_2^*}$, then $\pi_{t-1}(S^c) < 1/2$ and hence $\pi_{t-1}(S) > 1/2.$

Therefore it suffices to consider $S \subset [n]$ such that $|S| \le \big(1 - \frac{1}{2\alpha_2^*}\big)n.$

Let $$B~\coloneqq~\left\{S : \frac{n}{\eta \log n} \leq |S| \leq \bigg(1 - \frac{1}{2\alpha_2^*}\bigg)n \right\}.$$ We consider all $S\in B$ such that $|S| = s$ for some $\frac{n}{\eta \log n} \le s \le (1 - \frac{1}{2\alpha_2^*})n$.
Note that $e_t(S,S^c) \sim \mathrm{Bin}(s(n-s),p)$ for each $s$. We observe that, on $\mathcal{E}$, $$\frac{s(n-s)p}{2} \ge \frac{s \eta \log n}{4 \alpha_2^*}.$$
There are $\binom{n}{s}$ sets $S \in B$ of size $s$, thus, using classical tail bounds for the binomial distribution (e.g. Lemma 2.8.5.~in~\cite{durrett_2006} with $z = p/2$),

\begin{align}\label{BSetProb}
\mathbf{P}\Big(\exists S \in B \text{ with } |S| = s, e_t(S,S^c) \leq \frac{s(n-s)p}{2}\Big) \leq \binom{n}{s}\exp\left(-s(n-s)\frac{\eta\log n}{8(n-1)}\right)
\end{align}
and since $n - s \geq \frac{1}{2\alpha_2^*}n \geq \frac{1}{2\alpha_2^*}(n-1)$ we can use $\binom{n}{s} \leq \frac{n^s}{s!} \leq n^ss^{-s}e^s$ to arrive at an upper bound of
\begin{align*}
(\ref{BSetProb}) &\leq \exp\left(-s\left[\frac{\eta \log n}{16\alpha_2^*} + \log(s/n) - 1\right]\right) \\
&\leq \exp\left(-\frac{n}{\eta\log n}\left[\frac{\eta \log n}{16\alpha_2^*} - \log(\eta\log n) - 1\right]\right),
\end{align*}
where we used twice in the second line that $s \geq \frac{n}{\eta \log n}.$ This goes to zero exponentially fast, in particular faster than $n^{-\beta-2}.$ Since it holds for every $s$, a union bound yields that
$$\mathbf{P}\Big(e_t(S,S^c) \ge \frac{|S| \eta\log n}{4\alpha_2^*}, \forall t, \forall S \in B\Big) \ge 1 - 2n^{-\beta}.$$
Combining this with the previous bound on $Q_t(S,S)$, $c_2 = 2\eta$, and with (\ref{firstPhitEstimate}) implies
\begin{equation}
\mathbf{P}\Big(\Phi_t(S) \ge \frac{\alpha_1^*}{32(\alpha_2^*)^2},\forall t, \forall S \in B\Big) \ge 1 - 5n^{-\beta}.
\end{equation}

It remains to show a similar bound for $e_t(S,S^c)$ for every $S \in A = \left\{S : 1 \leq |S| \leq \frac{n-1}{\eta\log n} \right\}.$ On the event $D_t$, we observe that for every $S \subset [n]$
$$e_t(S,S^c) = \sum_{x \in S} \deg_t(x) - e_t(S,S) \geq |S|\cdot c_1\log n - e_t(S,S).$$
So in order to arrive at a lower bound for $e_t(S,S^c)$, we prove an upper bound on $e_t(S,S)$.

A priori, note that for $|S| = s \leq \frac{n-1}{\eta\log n}$, we have $\mathbf{E}[e_t(S,S)] \leq \frac{s^2}{2}p \leq \frac{s}{2}$. As in the case $S \in B$, we use standard tail bounds for the binomial distribution to estimate
\begin{align}
\begin{split}\label{SinAestimate}
\mathbf{P}(\exists S \in A \text{ with } |S| = s&, e_t(S,S) \geq s \log\log n) \leq C \binom{n}{s}p^{s \log\log n} \binom{s^2/2}{s \log\log n} \\
&\leq C \left(\frac{ne}{s}\right)^s \frac{(s^2/2)^{s \log\log n}p^{s \log\log n}}{(s \log\log n)^{s \log\log n}e^{s \log\log n}} \\
&= C \exp(q(s))
\end{split}
\end{align}
for
$q(s) \coloneqq s\log(ne/s)+ s\log\log n\left[\log s + \log\left(\frac{\eta\log ne}{2(n-1) \log\log n }\right)\right].$
We show that $q(s)$ is maximized when $s = 1$.
To prove this, we take the derivative with respect to $s$ and find
\begin{multline*}
q'(s) = \log(n) - \log(s)\,+ \\ \log\log n\big[\log s + 
\log(e\eta\log n)-\log\big(2(n-1)\log\log n\big)\big] + \log\log n.
\end{multline*}
Differentiating with respect to $s$ once more yields
$$q''(s) = \frac{\log\log n - 1}{s} > 0,$$ so $q'(s)$ is increasing in $s$.
In particular, we show that for the maximal $s = \frac{n-1}{\eta \log n}$ we have $q'(s) < 0$ which then implies that $q(s)$ is decreasing in $s$, hence $q(1)$ is the maximum on $1 \leq s \leq \frac{n-1}{\eta \log n}.$
We compute
\begin{align*}
q'\left(\frac{n-1}{\eta \log n}\right) &= \log\left(\frac{n}{n-1}\right) + \log(\eta\log n)+\log\log n\big[2 - \log(2) - \log\log\log n\big] \\
&= \log\left(\frac{n}{n-1}\right) + \log(\eta)+\log\log n\big[3 - \log(2) - \log\log\log n\big],
\end{align*}
which is smaller than 0 for $n$ large enough. Thus $q(s) \le q(1)$ for every $s$ and hence
\begin{align*}
q(s) &\leq \log(ne) + \log\log n \cdot \log\left(\frac{\eta\log ne}{2 \log\log n (n-1)}\right) \\
&= 1+\log n+\log\log n\left[1 + \log(\eta) + \log\log n - \log(2(n-1)) - \log\log\log n\right] \\
&\leq-\log n\cdot[\log\log n-1]+1+\log\log n\cdot\left[1 + \log(\eta) + \log\log n\right] \\
&\leq-\log n\cdot[\log\log n-2]
\end{align*}
for $n$ large enough, so ultimately
\begin{equation}
C \exp(q(s)) \leq C n^{2 - \log\log n} \leq n^{-\beta-2}
\end{equation}
for $n$ large enough. 

In particular, using a union bound for all $s \le \tfrac{n-1}{\eta \log n}$ on (\ref{SinAestimate}), results in the lower bound on the probability
\begin{equation}
\mathbf{P}(e_t(S,S) \le |S| \log \log n, \forall t, \forall S \in A) \ge 1 - n^{-\beta - 1}.
\end{equation}
Since $e_t(S,S^c) \ge |S| c_1 \log n - e_t(S,S)$, it follows that 
$$e_t(S,S^c) \ge |S| (c_1 \log n - \log \log n) \ge |S| \frac{c_1\log n}{2}.$$ 
Combining all of the above yields
$$\mathbf{P}\Big(\Phi_t(S) \ge \frac{\alpha_1^*c_1}{8\alpha_2^* c_2}, \forall t, \forall S \in A \Big) \ge 1 - 5n^{-\beta}.$$
Observing that, for our particular choice of constants, 
$$\frac{\alpha_1^*c_1}{8\alpha_2^* c_2} >  \frac{\alpha_1^*}{32(\alpha_2^*)^2}$$
that is to say $\Phi_t(S)$ can be smaller for $S \in B$ than for $S \in A$, it follows that
\begin{equation}
\mathbf{P}\Big(\Phi_t^* \ge \frac{\alpha_1^*}{32(\alpha_2^*)^2}, \forall t\Big) \ge 1 - 10n^{-\beta}.
\end{equation}
Furthermore on $\mathcal{E}$, it holds that $\sqrt{\pi_0^{\mathrm{min}}\pi_t^{\mathrm{min}}} \geq \frac{\alpha_1^*}{n}$ and
$\left(\frac{1}{2} \frac{g_s}{1-\frac{1}{2}g_s}\right)^2 \geq \left(\frac{\alpha_1^*}{2\alpha_2^*-\alpha_1^*}\right)^2$. Hence 
\begin{equation}
\Theta_t = \min_{1 \leq s \leq t}\bigg(\frac{1}{2} \frac{g_s}{1-\frac{1}{2}g_s}\Phi^*_s\bigg)^2 \geq \bigg(\frac{\alpha_1^*}{2\alpha_2^*-\alpha_1^*}\cdot \frac{\alpha_1^*}{32(\alpha_2^*)^2}\bigg)^2
\end{equation}
with high probability and~(\ref{thetaBound}) follows.
\end{proof}
Theorem~\ref{logUpperBound} is an easy consequence:
\begin{proof}[Proof of Theorem~\ref{logUpperBound}]
From Proposition~\ref{PiConcProp} giving a lower bound on $\pi_t^{\mathrm{min}}$ and Proposition~\ref{thetaBoundProp} providing a lower bound on $\Theta_t$, the assumptions of Corollary $\ref{probabilisticMainThCor}$ are satisfied, which implies (\ref{logUpperBoundEq}).
\end{proof}

\subsection{A lower bound on mixing time}

We already know from Proposition~\ref{PiConcProp} that $\pi_t(x)~\in~[\alpha_1^*/n, \alpha_2^*/n]$ for all $x \in [n]$, $t \in \{0, ..., n\}$ w.h.p., which also allows us to employ a classical argument for a lower bound on the mixing time. Namely, a random walk cannot mix if a substantial section of the graph is not accessible. This also shows that the intuition (see Remark~\ref{falseIntuition}) that the walk should mix in just a few steps is wrong, as long as $p$ is not too large.

\begin{theorem}\label{logLowerBound}
Let $\eta > 50$ independent of $n$, $p = \frac{\eta \log n}{n-1}.$ 
Let $\varepsilon \in (0, 1/2)$. There exists $c > 0$ such that
$$t_\mathrm{mix}(\varepsilon,0) \geq c \frac{\log n}{\log\log n}$$
with high probability.
\end{theorem}
\begin{proof}
As in the proof of Proposition~\ref{thetaBoundProp}, we note:
If $$|S| > \left(1 - \frac{1}{2\alpha_2^*}\right)n$$
then (w.h.p.)
$$\pi_t(S) > \frac{1}{2}$$ for all $0 \leq t \leq n$ . In particular, if only vertices in $S^c$ are reachable within $t$ steps from $x$, then 
$$\|P^{0,t}(x, \cdot) - \pi_t(\cdot)\|_{\mathrm{TV}} \geq 1 - \pi_t(S^c) = \pi_t(S) > \frac{1}{2}.$$
Let $T^{k}_x$ be the set of vertices reachable from $x$ in $k$ steps starting at time 0, in the sense that $T^{k}_x = \{ y \in [n] : \text{there exists a path }[x,y]_0^k\}$ (see the proof of Lemma~\ref{ascentLemmaHelper}). Since every vertex (w.h.p.) has at most $c_2 \log n$ neighbors,
$$|T^1_x| \leq c_2 \log n + 1.$$
Even if all those vertices get an entirely new set of neighbors, the maximum number of reachable vertices after two steps is still bounded by
$|T^2_x| \leq (c_2 \log n + 1)^2.$ Iteratively $|T^k_x| \leq (c_2 \log n + 1)^k.$ For $k$ small enough, this implies that the random walk is confined to a small set of vertices. In particular, if
$$k \leq \frac{\log\left(\frac{1}{2\alpha_2^*}n \right)}{\log\left(1 + c_2 \log n \right)},$$ then
$|T^k_x| \leq n/(2\alpha_2^*)$ and thus $d(0,t) > 1/2$ with high probability.
We can choose $c > 0$ such that $$t_\mathrm{mix}(\varepsilon,0) \geq c \frac{\log n}{\log\log n}$$
with high probability.
\end{proof}

\appendix

\section{Existence proofs}\label{appendix:existence}

In order to prove Lemma~\ref{Lemmaexist}, we need the following simple observation on total variation. 
\begin{lemma}\label{convexityLemma}
Let $$\hat{d}(s,t) \coloneqq \sup_{\mu,\nu} \|\mu P^{s,t} - \nu P^{s,t}\|_\mathrm{TV}$$
where the supremum is taken over all probability measures on $[n]$. Then the supremum is attained, and in fact it holds that
$$\hat{d}(s,t) = \delta(P^{s,t}) = \sup_{x,y} \|P^{s,t}(x,\cdot) - P^{s,t}(y,\cdot)\|_\mathrm{TV}.$$
\end{lemma}
\begin{proof}See Exercise 4.1 in~\cite{LevinPeresWilmer2006} for the statement, and Appendix D therein for the proof.
\end{proof}

\begin{proof}[Proof of Lemma~\ref{Lemmaexist}]
Let $\varepsilon > 0$. By assumption, there exists $u < t$, such that $\delta(P^{u,t}) \leq \varepsilon/2.$
Choose any $s,r \in \mathbb{Z}$ with $s < r < u$.
Then, denoting $\mu_x = P^{s,u}(x,\cdot)$ and $\nu_y = P^{r,u}(y,\cdot),$
\begin{align*}
\sup_{x,y} \|P^{s,t}(x,\cdot) - P^{r,t}(y,\cdot)\|_\mathrm{TV} &= \sup_{x,y} \|(P^{s,u}P^{u,t} )(x,\cdot) - (P^{r,u}P^{u,t})(y,\cdot)\|_\mathrm{TV} \\
&= \sup_{x,y} \|\mu_xP^{u,t} - \nu_yP^{u,t}\|_\mathrm{TV} \\
&\leq \sup_{x,y} \|P^{u,t}(x,\cdot) - P^{u,t}(y,\cdot)\|_\mathrm{TV} \\
&= \delta(P^{u,t}) \leq \varepsilon/2.
\end{align*}
where the inequality in the third line follows from Lemma~\ref{convexityLemma}.
In particular, for any $x,y \in [n]$ we have 
$$\sum_z |P^{s,t}(x,z) - P^{r,t}(y,z)| \leq \varepsilon,$$
so the matrix products $(P^{s,t})_{s \leq t}$ are a Cauchy sequence as $s \to -\infty$. By completeness of the space of matrices, the Cauchy sequence converges and thus there exists a matrix $Q^t$ with $\lim_{s \to -\infty} P^{s,t}(x,y) = Q^t(x,y).$ Moreover $\delta(Q^t) = 0$, so it is a rank 1 matrix.
\end{proof}
We now show the technical results of Lemma~\ref{lemma3}, that $d(s,t)$ is decreasing in $t$ and that it is closely related to $\delta(P^{s,t}).$
\begin{proof}[Proof of Lemma~\ref{lemma3}] To show (a), simply consider 
\begin{align*}
\sup_{x} \|P^{u,t}(x,\cdot) - \pi_{t}(\cdot)\|_\mathrm{TV} &= \sup_{x}\|(P^{u,s}P^{s,t})(x,\cdot) - (\pi_sP^{s,t})(\cdot)\|_\mathrm{TV} \\
&\le \sup_{x} \|P^{u,s}(x,\cdot) - \pi_s(\cdot)\|_\mathrm{TV} 
\end{align*}
by the definition of the total variation distance and applying the triangle inequality. So $d(u,s) \le d(u,t)$ for all $u \le s \le t$.

(b) is shown by using the same arguments as in Lemma~\ref{convexityLemma} (see Exercise 4.1. in Appendix~D of~\cite{LevinPeresWilmer2006}). We omit the details here because the statement is not used in the present paper. 

To show (c), remark that $d(s,t) \le \delta(P^{s,t})$ (see again Lemma~\ref{convexityLemma}). To finally prove  $\delta(P^{s,t})~\le~2d(s,t)$, write
\begin{align*}
\sup_{x,y} \sum_{z} |P^{s,t}(x,z) - P^{s,t}(y,z)| &= \sup_{x,y} \sum_{z} |P^{s,t}(x,z) - \pi_t(z) + \pi_t(z) - P^{s,t}(y,z)| \\
&\le \sup_x \sum_{z} |P^{s,t}(x,z) - \pi_t(z)| + \sup_y \sum_{z}|P^{s,t}(y,z) - \pi_t(z)|
\end{align*}
which completes the proof.
\end{proof}

\section{Erd\H{o}s-R\'enyi graphs}\label{appendix:concentration}
In this section, we first show Lemma~\ref{degreeLemma}, giving degree bounds for connected Erd\H{o}s-R\'enyi graphs. Similar results are standard (cf.~Lemma 6.5.2 in~\cite{durrett_2006}), but here we compute an explicit decay rate of the right-hand side for our choice of parameters.
\begin{proof}[Proof of Lemma~\ref{degreeLemma}] Note that $\mathrm{deg}_t(x)\sim \mathrm{Bin}(n-1,p).$ For any $\lambda > 0$, $c_2 > \eta$
\begin{align*}
\mathbf{P}(\mathrm{deg}_t(x) \ge c_2 \log n) &\le e^{\log(1-p+pe^\lambda)(n-1)}n^{-\lambda c_2} \le e^{(-p+pe^\lambda)(n-1)}n^{-\lambda c_2} \\
&= n^{-\eta + \eta e^\lambda - \lambda c_2}.
\end{align*}
We can minimize this expression by choosing $\lambda = \log (\frac{c_2}{\eta})$ to arrive at
\begin{equation}
\mathbf{P}(\mathrm{deg}_t(x) \ge c_2 \log n) \le n^{-\eta + c_2(1 - \log(\frac{c_2}{\eta}))}.
\end{equation}
Similarly, for $0 < c_1 < \eta$, we can establish the lower bound by
\begin{align*}
\mathbf{P}(\mathrm{deg}_t(x) \le c_1 \log n) \le n^{-\eta + \eta e^{-\lambda}+\lambda c_1},
\end{align*}
which we can optimize with $\lambda = \log(\frac{\eta}{c_1})$ to arrive at
\begin{equation}
\mathbf{P}(\mathrm{deg}_t(x) \le c_1 \log n) \le n^{-\eta + c_1(1 + \log(\frac{\eta}{c_1}))}.
\end{equation}
Since $c_1 = \frac{11}{21} \eta$, $c_2 = 2\eta$ and $\eta > 50$, applying a union bound over all $n$ vertices yields the desired result.
\end{proof}
We now show Lemma~\ref{connectedGraphLemma}, that $n^2$ independent Erd\H{o}s-R\'enyi graphs are all connected with high probability and again compute the decay rate.
\begin{proof}[Proof of Lemma~\ref{connectedGraphLemma}] There are countless proofs of connectivity (for example Theorem 2.8.1. in~\cite{durrett_2006}). The argument presented here specifically uses that the graphs are far above the connectivity threshold (recall $\eta > 50$ by assumption) to get a quantitative bound on the probability.

We consider the connectivity at a time $k \in \{0, \dots, n^2 - 1\}$. Let $K$ be the number of ways the graph can be divided at time $k$ into two non-empty subgraphs that are disconnected. Then clearly $\mathbf{P}(\text{Graph not connected}) = \mathbf{P}(K \ge 1) \le \mathbf{E}[K]$.
We estimate
\begin{align*}
\mathbf{E}[K] &= \sum_{i = 1}^{n/2} (1 - p)^{i(n-i)}\binom{n}{i} \le \sum_{i=1}^{n/2}e^{-pi(n-i)}n^i \\
&= \sum_{i=1}^{n/2} e^{-pi(n-i)+i\log n} = \sum_{i=1}^{n/2}e^{(-\frac{\eta}{n-1}i(n-i)+i)\log n} \\
&\le \sum_{i=1}^{n/2} e^{(-\eta i(\frac{n}{n-1} - \frac{n}{2(n-1)})+i)\log n} = \sum_{i=1}^{n/2}n^{-i(\frac{\eta}{2}\frac{n}{n-1} - 1)} \\
&\le n^{-\frac{\eta}{2}+2}.
\end{align*}
Union bound for all $n^2$ times yields the result.
\end{proof}
\bibliographystyle{amsplain}
\bibliography{basics}

\end{document}